\documentclass[10pt]{amsart}
\usepackage{amssymb,amsmath,amsthm,amscd}
\usepackage{leftidx,breqn}
\usepackage{graphicx}
\usepackage{color}
\usepackage{float}
\usepackage{hyperref}
\usepackage[usenames,dvipsnames,svgnames,table]{xcolor}
\usepackage{tikz}
\usepackage{comment}
\usepackage{mathtools}
\usepackage{epstopdf} 
\usepackage{tikz-cd}
\newtheorem{lemma}{Lemma}[section]
\newtheorem{theorem}[lemma]{Theorem}
\newtheorem{corollary}[lemma]{Corollary}
\newtheorem{proposition}[lemma]{Proposition}

\theoremstyle{definition}
\newtheorem{definition}[lemma]{Definition}
\theoremstyle{remark}
\newtheorem{remark}[lemma]{Remark}
\newcommand{\C}{\mathbb{C}}

\newcommand{\N}{\mathbb{N}}

\newcommand{\R}{\mathbb{R}}

\newcommand{\cC}{\mathcal{C}}
\newcommand{\cH}{\mathcal{H}}
\newcommand{\cM}{\mathcal{M}}
\newcommand{\cR}{\mathcal{R}}

\DeclareMathOperator{\re}{Re}
\DeclareMathOperator{\im}{Im}
\DeclareMathOperator{\Int}{int}

\DeclareMathOperator{\per}{Per}

\renewcommand{\epsilon}{\varepsilon}
\renewcommand{\phi}{\varphi}
\renewcommand{\theta}{\vartheta}

\DeclareMathOperator{\Deg}{deg}
\DeclareMathOperator{\Diff}{Diff}

\DeclareMathOperator{\Poly}{Poly}

\title[Discontinuity of straightening]{Discontinuity of straightening in anti-holomorphic dynamics: II\\}

\author[H.~Inou]{Hiroyuki Inou}
\address{Department of Mathematics, Kyoto University, Kyoto 606-8502, Japan}
\email{inou@math.kyoto-u.ac.jp}

\author[S.~Mukherjee]{Sabyasachi Mukherjee}
\address{School of Mathematics, Tata Institute of Fundamental Research, 1 Homi Bhabha Road, Mumbai 400005, India}
\email{sabya@math.tifr.res.in} 

\date{\today}

\begin{document}

\begin{abstract}
In \cite{M3}, Milnor found Tricorn-like sets in the parameter space of real cubic polynomials. We give a rigorous definition of these Tricorn-like sets as suitable renormalization loci, and show that the dynamically natural straightening map from such a Tricorn-like set to the original Tricorn is discontinuous. We also prove some rigidity theorems for polynomial parabolic germs, which state that one can recover unicritical holomorphic and anti-holomorphic polynomials from their parabolic germs.

\end{abstract}

\maketitle

\setcounter{tocdepth}{1}

\tableofcontents

\section{Introduction}

The Tricorn, which is the connectedness locus of quadratic anti-holomorphic polynomials, is the anti-holomorphic counterpart of the Mandelbrot set. 

\begin{definition}[Multicorns]
The \emph{multicorn} of degree $d$ is defined as $$\mathcal{M}^{\ast}_d = \lbrace c \in \mathbb{C} : K(f_c)\ \mathrm{is\ connected}\rbrace,$$ where $K(f_c)$ is the filled Julia set of the unicritical anti-holomorphic polynomial $f_c(z)=\overline{z}^d+c$. The multicorn of degree $2$ is called the \emph{Tricorn}.
\end{definition} 

While it follows from classical works of Douady and Hubbard that baby Mandelbrot sets are homeomorphic to the original Mandelbrot set via dynamically natural straightening maps \cite{DH2}, it was shown in \cite{IM4} that straightening maps from Tricorn-like sets appearing in the Tricorn to the original Tricorn are discontinuous at infinitely many parameters. 

The dynamics of quadratic anti-holomorphic polynomials and its connectedness locus, the Tricorn, was first studied in \cite{CHRS}, and their numerical experiments showed major structural differences between the Mandelbrot set and the Tricorn. Nakane proved that the Tricorn is connected, in analogy to Douady and Hubbard's classical proof of connectedness of the Mandelbrot set \cite{Na1}. Later, Nakane and Schleicher, in \cite{NS}, studied the structure of hyperbolic components of the multicorns, and Hubbard and Schleicher \cite{HS} proved that the multicorns are not pathwise connected. Recently, in an attempt to explore the topological aspects of the parameter spaces of unicritical anti-holomorphic polynomials, the combinatorics of external dynamical rays of such maps were studied in \cite{Sa} in terms of orbit portraits, and this was used in \cite{MNS} where the bifurcation phenomena and the structure of the boundaries of hyperbolic components were described. The authors showed in \cite{IM} that many parameter rays of the multicorns non-trivially accumulate on persistently parabolic regions. For a brief survey of anti-holomorphic dynamics and associated parameter spaces (in particular, the Tricorn), we refer the readers to \cite[\S 2]{LLMM2}.

It was Milnor who first identified multicorns as prototypical objects in parameter spaces of \emph{real} rational maps (here, a rational map if called \emph{real} if it commutes with an anti-holomorphic involution of the Riemann sphere) \cite{M3,M4}. In particular, he found Tricorn-like sets in the connectedness locus of real cubic polynomials. One of the main goals of the current paper is to explain the appearance of Tricorn-like sets in the real cubic locus, and to prove that straightening maps from these Tricorn-like sets to the original Tricorn is always discontinuous. 

Let us now briefly illustrate how quadratic anti-polynomial-like behavior can be observed in the dynamical plane of a real cubic polynomial (see \cite[Definition~5.1]{IM4} for the definition of anti-quadratic-like maps). According to \cite{M3}, a hyperbolic component $H$ in the parameter space of cubic polynomials is said to be \emph{bitransitive} if each polynomial $p$ in $H$ has a unique attracting cycle (in $\C$) such that the two distinct critical points of $p$ lie in two different components of the immediate basin of the attracting cycle. In particular, associated to each bitransitive component $H$ there are two positive integers $n_1, n_2$ such that if $p$ is the center of $H$ with (distinct) critical points $c_1$ and $c_2$, then $p^{\circ n_1}(c_1)=c_2$ and $p^{\circ n_2}(c_2)=c_1$. Now let 
$$
p(z) = -z^3-3a_0^2z+b_0,\ a_0\geq 0, b_0 \in \mathbb{R},
$$ 
be a real cubic polynomial that is the center of a bitransitive hyperbolic component. The two critical points $c_1:=ia_0$ and $c_2:=-ia_0$ of $p$ are complex conjugate, and have complex conjugate forward orbits. Thus, the assumption that $p$ is the center of a bitransitive hyperbolic component implies that there exists an $n \in \mathbb{N}$ such that $p^{\circ n}(c_1) = c_2,\ p^{\circ n}(c_2) = c_1$ (compare Figure~\ref{real cubic}). Suppose $U$ is a neighborhood of the closure of the Fatou component containing $c_1$ such that $p^{\circ 2n}: U\rightarrow p^{\circ 2n}(U)$ is polynomial-like of degree $4$. Then we have, $\iota(\overline{U}) \subset p^{\circ n}(U)$ (where $\overline{U}$ is the topological closure of $U$, and $\iota$ is the complex conjugation map), i.e.,\ $\overline{U} \subset (\iota \circ p^{\circ n})(U)$. Therefore $\iota \circ p^{\circ n}:U\rightarrow (\iota \circ p^{\circ n})(U)$ is a proper anti-holomorphic map of degree $2$, hence an anti-polynomial-like map of degree $2$ (with a connected filled Julia set) defined on $U$. An anti-holomorphic version of the straightening theorem \cite[Theorem~5.3]{IM4} now yields a quadratic anti-holomorphic map (with a connected filled Julia set) that is hybrid equivalent to $(\iota \circ p^{\circ n})\vert_U$. One can continue to perform this renormalization procedure as the real cubic polynomial $p$ moves in the parameter space, and this defines a map from a suitable region in the parameter plane of real cubic polynomials to the Tricorn. We will define these Tricorn-like sets rigorously as suitable renormalization loci $\mathcal{R}(a_0,b_0)$, and will define the dynamically natural `straightening map' from $\mathcal{R}(a_0,b_0)$ to the Tricorn $\mathcal{M}_2^*$ in Section~\ref{proof_main_thm}.  

\begin{theorem}[Discontinuity of Straightening in Real Cubics]\label{Straightening_discontinuity_2}
The straightening map $\chi_{a_0,b_0}: \mathcal{R}(a_0,b_0) \to \mathcal{M}_2^*$ is discontinuous (at infinitely many explicit parameters).
\end{theorem}

The study of straightening maps in anti-holomorphic dynamics was initiated in \cite{I1,IM4}. In \cite[Theorem~1.1]{IM4}, the authors proved discontinuity of straightening maps for Tricorn-like sets contained in the Tricorn by demonstrating `wiggling of umbilical cords' for non-real odd-periodic hyperbolic components of the Tricorn (compare \cite[Theorem~1.2]{IM4}). 

The ideas that go into the proof of Theorem~\ref{Straightening_discontinuity_2} are similar to those used to prove \cite[Theorem~1.1, Theorem~1.2]{IM4}. More precisely, the proof of discontinuity is carried out by showing that the straightening map from a Tricorn-like set in the real cubic locus to the original Tricorn sends certain `wiggly' curves to landing curves. Indeed, there exist hyperbolic components $H$ of the Tricorn such that $H$ intersects the real line, and the `umbilical cord' of $H$ lands on the root parabolic arc on $\partial H$. In other words, such a component can be connected to the period $1$ hyperbolic component by a path. On the other hand, using parabolic implosion arguments and analytic continuation of local analytic conjugacies, we show that the \emph{non-real} umbilical cords for the Tricorn-like sets (in the real cubic locus) do not land at a single point. Discontinuity of straightening maps now follows from the observation that (the inverse of) straightening maps send suitable landing umbilical cords to wiggly umbilical cords. 

For topological properties of straightening maps in more general polynomial parameter spaces and the associated discontinuity phenomena, we encourage the readers to consult \cite{IK,I}. The main difference between the discontinuity results proved in \cite{I} and the current paper is that the proof of discontinuity appearing in \cite{I} strictly uses complex two-dimensional bifurcations in the parameter space, and hence cannot be applied to prove discontinuity of straightening maps in real two-dimensional parameter spaces (such as the parameter space of real cubic polynomials). On the other hand, the present proof employs a one-dimensional parabolic perturbation argument to prove `wiggling of umbilical cords' (explained in the previous paragraph), which leads to discontinuity of straightening maps.

\begin{figure}[ht!]
\begin{minipage}{0.45\linewidth}
\centering{\includegraphics[width=0.96\linewidth]{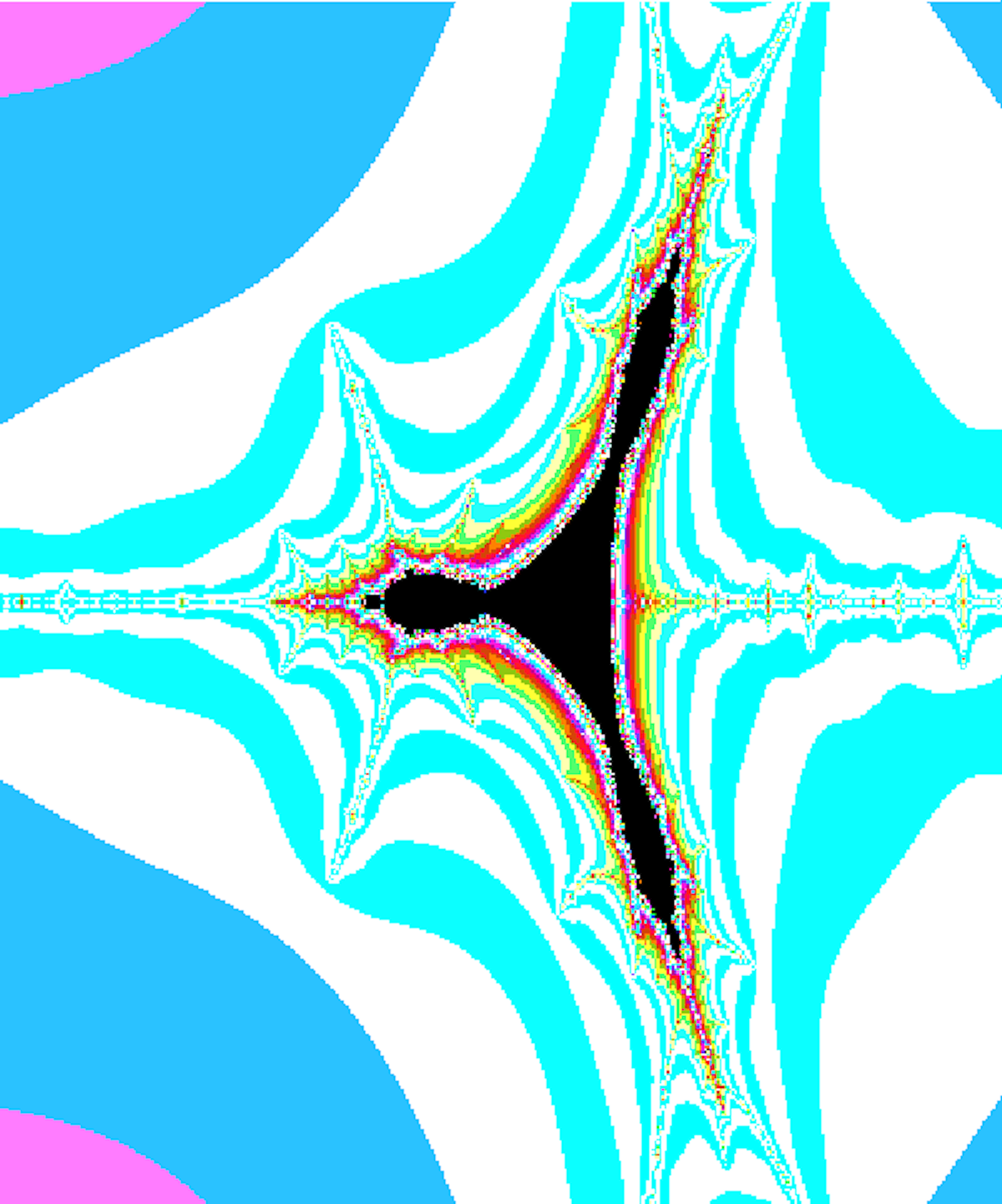}}
\end{minipage}
\begin{minipage}{0.54\linewidth} 
\centering{\includegraphics[width=0.96\linewidth]{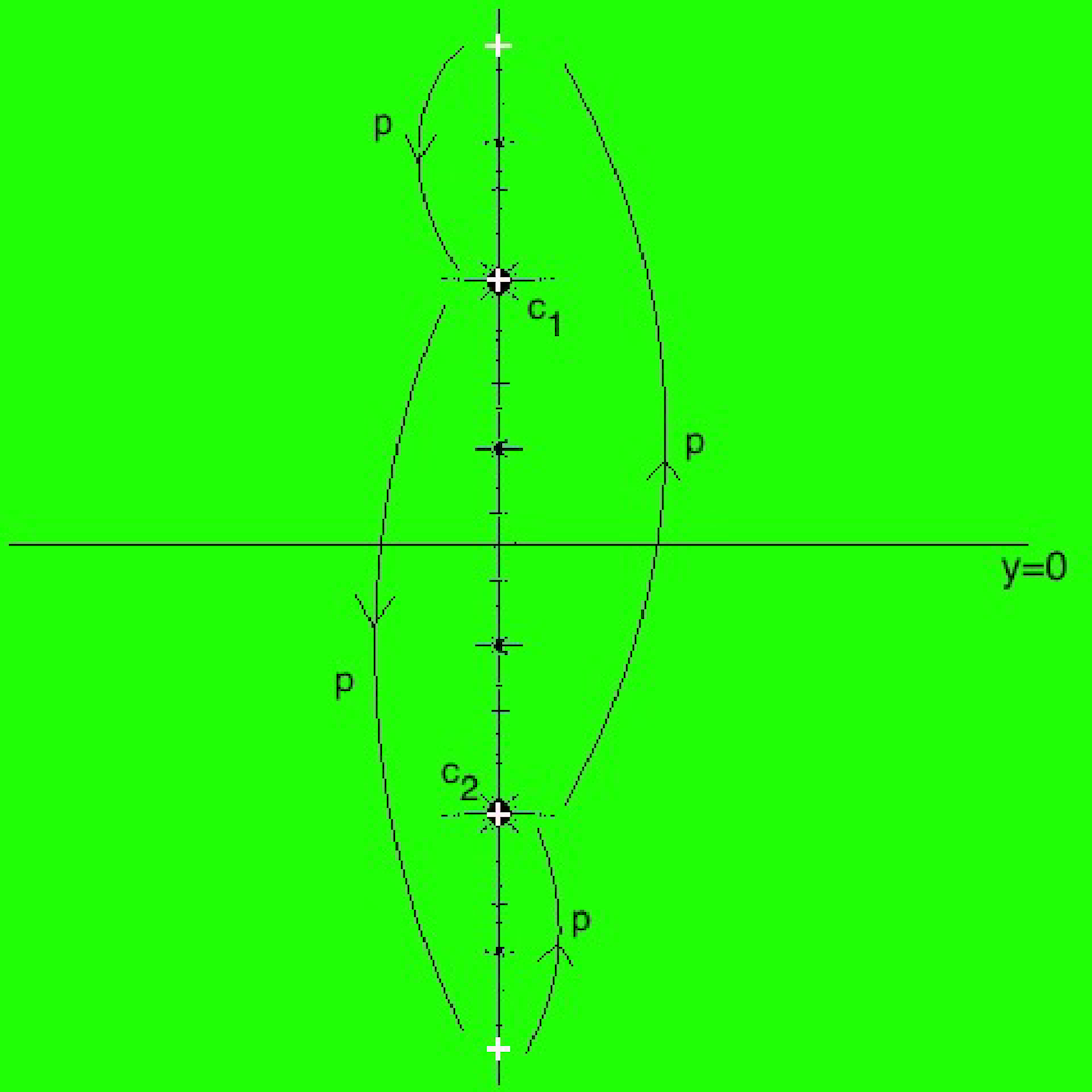}}
\end{minipage} 
\caption{Left: A Tricorn-like set in the parameter plane of real cubic polynomials. Right: The dynamical plane of the center of a bitransitive hyperbolic component in the parameter space of real cubic polynomials with real-symmetric critical points.} 
\label{real cubic} 
\end{figure}

One of the key steps in the proof of `wiggling of umbilical cords' in this paper as well as in \cite{IM4} is to extend a carefully constructed local conjugacy between parabolic germs (of two polynomials) to a semi-local conjugacy (i.e.,\ conjugacy between polynomial-like maps), which allows us to conclude that the corresponding polynomials are affinely conjugate. In general, we believe that two polynomial parabolic germs can be conformally conjugate only if the two polynomials are polynomial semi-conjugates of a common polynomial. Motivated by these considerations, we prove some rigidity principles for unicritical holomorphic and anti-holomorphic polynomials with parabolic cycles. These questions are of independent interest in the theory of parabolic germs. 

Let $\mathcal{M}_d^{\mathrm{par}}$ be the set of parabolic parameters of the multibrot set $\mathcal{M}_d$, which is the connectedness locus of unicritical holomorphic polynomials $p_c(z)=z^d+c$ (see \cite{EMS} for background on the multibrot sets). For a polynomial $p_c$ with a parabolic cycle, we define the \emph{characteristic Fatou component} of $p_c$ as the unique Fatou component of $p_c$ containing the critical value $c$. The \emph{characteristic parabolic point} of $p_c$ is defined as the unique parabolic point on the boundary of the characteristic Fatou component. We prove the following theorem.

\begin{theorem}[Parabolic Germs Determine Roots and Co-roots of Multibrot Sets]\label{Parabolic_Germs_Determine_Roots_Co_Roots}
For $i=1, 2$, let $c_i\in \mathcal{M}_{d_i}^{\mathrm{par}}$, $z_i$ be the characteristic parabolic point of $p_{c_i}(z)=z^{d_i}+c_i$, and $n_i$ be the period of the characteristic Fatou component of $p_{c_i}$. If the restrictions $p_{c_1}^{\circ n_1}\vert_{N_{z_1}}$ and $p_{c_2}^{\circ n_2}\vert_{N_{z_2}}$ (where $N_{z_i}$ is a sufficiently small neighborhood of $z_i$) are conformally conjugate, then $d_1=d_2$, and $p_{c_1}$ and $p_{c_2}$ are affinely conjugate.
\end{theorem}

It is worth mentioning that in the case when $d_1=d_2$, the conclusion of Theorem~\ref{Parabolic_Germs_Determine_Roots_Co_Roots} also follows from \cite[Theorem~1.1]{LoMu}. However, the proof of this result given in \cite{LoMu} uses the language of \emph{parabolic-like maps} (thereby establishing a rigidity result for suitable parabolic-like maps), while the proof given in the current paper only employs the more classical machinery of polynomial-like maps. 

To formulate the anti-holomorphic analogue of the previous theorem, let us define $\Omega_d^{\mathrm{odd}}:=\{ c\in \mathbb{C}: f_c(z)=\overline{z}^d+c$ has a parabolic cycle of \emph{odd} period with a single petal$\}$, and $\Omega_d^{\mathrm{even}}:=\{ c\in \mathbb{C}: f_c(z)=\overline{z}^d+c$ has a parabolic cycle of \emph{even} period$\}$. For $c_1, c_2 \in \Omega_d^{\mathrm{odd}}\cup\Omega_d^{\mathrm{even}}$, we write $c_1 \sim c_2$ if $\overline{z}^d+c_1$ and $\overline{z}^d+c_2$ are affinely conjugate; i.e.,\ if $c_2/c_1$ is a $(d+1)$-st root of unity. We denote the set of equivalence classes under this equivalence relation by $\Omega_d^{\mathrm{odd}}\cup\Omega_d^{\mathrm{even}}/\mathord\sim$. By abusing notation, we will identify $c_i$ with its equivalence class in $\Omega_d^{\mathrm{odd}}\cup\Omega_d^{\mathrm{even}}/\mathord\sim$. The first obstruction to recovering $f_c$ from its parabolic germ comes from the following observation: if $c\in \Omega_d^{\mathrm{odd}}$ has a parabolic cycle of odd period $k$, then the characteristic parabolic germs of $f_c^{\circ 2k}$ and $f_{c^*}^{\circ 2k}$ are conformally conjugate by the map $\iota\circ f_{c}^{\circ k}$ (here, and in the sequel, $z^*$ will stand for the complex conjugate of the complex number $z$). The next theorem shows that this is, in fact, the only obstruction.

\begin{theorem}[Recovering Anti-polynomials from Their Parabolic Germs]\label{recovering_Anti-polynomials}
For $i=1, 2$, let $c_i \in\left(\Omega_{d_i}^{\mathrm{odd}}\cup\Omega_{d_i}^{\mathrm{even}}/\mathord\sim\right)$, $z_i$ be the characteristic parabolic point of $f_{c_i}$, and $n_i$ be the period of the characteristic Fatou component of $f_{c_i}$ under $f_{c_i}^{\circ 2}$. If the parabolic germs $f_{c_1}^{\circ 2n_1}$ and $f_{c_2}^{\circ 2n_2}$ around $z_1$ and $z_2$ (respectively) are conformally conjugate, then $d_1=d_2 = d$ (say), and one of the following is true.
\begin{enumerate}
 
\item $c_1, c_2\in \Omega_{d}^{\mathrm{even}}$, and $c_1=c_2$ in $\Omega_d^{\mathrm{even}}/\mathord\sim$.
 
\item $c_1, c_2\in \Omega_{d}^{\mathrm{odd}}$, and $c_2\in \lbrace c_1, c_1^*\rbrace$ in $\Omega_d^{\mathrm{odd}}/\mathord\sim$.

\end{enumerate}
\end{theorem}

A parabolic germ $g$ at $0$ is said to be \emph{real-symmetric} if in some conformal coordinates, we have $$\overline{g(\overline{z})}=g(z);$$ i.e.,\ if all the coefficients in its power series expansion are real after a local conformal change of coordinates. We prove that the tangent-to-identity parabolic germ of a unicritical parabolic polynomial is real-symmetric if and only if the polynomial commutes with an anti-holomorphic involution of the plane.

\begin{theorem}[Real-symmetric Germs Only for Real Polynomials]
\label{Real_Germs_Real_Parameters}
Let $c$ be in $\mathcal{M}_d^{\mathrm{par}}$, $z_c$ be the characteristic parabolic point of $p_c$, $U_c$ be the characteristic Fatou component of $p_c$, and $n$ be the period of the component $U_c$. If the parabolic germ of $p_c^{\circ n}$ at $z_c$ is real-symmetric, then $p_c$ commutes with an anti-holomorphic involution of the plane.
\end{theorem}

On the anti-holomorphic side, we have the following result.

\begin{theorem}[Real-symmetric Germs Only for Real Anti-polynomials]
\label{Real_Germs_Real_Parameters_Anti}
Let $c$ be in $\Omega_{d}^{\mathrm{odd}}\cup \Omega_{d}^{\mathrm{even}}$, $z_c$ be the characteristic parabolic point of $f_c$, $U_c$ be the characteristic Fatou component of $f_c$, and $n$ be the period of the component $U_c$ under $f_c^{\circ 2}$. If the parabolic germ of $f_c^{\circ 2n}$ at $z_c$ is real-symmetric, then $f_c$ commutes with an anti-holomorphic involution of the plane.
\end{theorem}

Let us now outline the organization of the paper. Section~\ref{Tricorns_Real_Cubics} is devoted to a study of Tricorn-like sets in the real cubic locus. After preparing the necessary background on Tricorn-like sets in Subsections~\ref{sec_per_one} and~\ref{sec_bitran}, we prove Theorem~\ref{Straightening_discontinuity_2} (which asserts that the straightening map from any Tricorn-like set in the real cubic locus to the original Tricorn is discontinuous) in Subsection~\ref{proof_main_thm}. The rest of the paper concerns local-global principles for polynomial parabolic germs. In Section~\ref{cauliflower_recover}, we first recall some known facts about extended horn maps, and give a proof of Theorem~\ref{Parabolic_Germs_Determine_Roots_Co_Roots} using the mapping properties of extended horn maps. In this section, we also prove Theorem~\ref{recovering_Anti-polynomials} to the effect that one can recover the parabolic parameters of the multicorns, up to some natural rotational and reflection symmetries, from their parabolic germs. Finally, in Section~\ref{real-symmetric_parabolic_germs}, we prove Theorems~\ref{Real_Germs_Real_Parameters} and~\ref{Real_Germs_Real_Parameters_Anti} which state that the parabolic germ of a unicritical holomorphic polynomial (respectively anti-polynomial) is conformally conjugate to a real-symmetric parabolic germ if and only if the polynomial (respectively anti-polynomial) commutes with a global anti-holomorphic involution whose axis of symmetry passes through the parabolic point. 
\bigskip

\noindent\textbf{Acknowledgements.} The first author would like to express his gratitude for the support of JSPS KAKENHI Grant Number 26400115. The second author gratefully acknowledges the support of Deutsche Forschungsgemeinschaft DFG, the Institute for Mathematical Sciences at Stony Brook University, and an endowment from Infosys Foundation during parts of the work on this project.

\section{Tricorns in Real Cubics}\label{Tricorns_Real_Cubics}
\textbf{A standing convention:} In the rest of the paper, we will denote the complex conjugate of a complex number $z$ either by $\overline{z}$ or by $z^*$. The complex conjugation map will be denoted by $\iota$, i.e.,\ $\iota(z)=z^*$. The image of a set $U$ under complex conjugation will be denoted as $\iota(U)$, and the topological closure of $U$ will be denoted by $\overline{U}$.

In this section, we will discuss some topological properties of Tricorn-like sets, and umbilical cords in the family of real cubic polynomials. We will work with the family:
$$
\mathcal{G} = \{ g_{a,b}(z)= -z^3 - 3a^2z+b,\ a \geq 0, b\in \mathbb{R}\}.
$$

Milnor \cite{M3} numerically found that the connectedness locus of this family contains Tricorn-like sets. We will rigorously define Tricorn-like sets in this family (via straightening of suitable anti-polynomial-like maps), and show that the corresponding straightening maps from these Tricorn-like sets to the original Tricorn are discontinuous. 

The map $g_{a,b}$ commutes with complex conjugation $\iota$; i.e.,\ $g_{a,b}$ has a reflection symmetry with respect to the real line.  Observe that $g_{a,b}$ is conjugate to the monic centered polynomial $h_{a,b}(z)= z^3-3a^2z+bi$ by the affine map $z\mapsto iz$, and $h_{a,b}$ has a reflection symmetry with respect to the imaginary axis (i.e.,\ $h_{a,b}(-\iota(z))=-\iota(h_{a,b}(z))$). We will, however, work with the real form\footnote{This parametrization has the advantage that the critical orbits are complex conjugate.} $g_{a,b}$, and will normalize the B{\"o}ttcher coordinate $\phi_{a,b}$ of $g_{a,b}$ (at $\infty$) such that $\phi_{a,b}(z)/z \to -i$ as $z \to \infty$. Roughly speaking, the invariant dynamical rays $R_{(a,b)}(0)$ and $R_{(a,b)}(1/2)$ tend to $+i\infty$ and $-i\infty$ respectively as the potential tends to infinity. By symmetry with respect to the real line, the $2$-periodic rays $R_{(a,b)}(1/4)$ and $R_{(a,b)}(3/4)$ are contained in the real line. Also note that $-g_{a,b}(-z)=g_{a,-b}(z)$. Nonetheless, to define straightening maps consistently, we need to distinguish $g_{a,b}$ and $g_{a,-b}$ as they have different rational laminations (with respect to our normalized B{\"o}ttcher coordinates). We will denote the connectedness locus of $\mathcal{G}$ by $\mathcal{C}(\mathcal{G})$.

\begin{remark}
The parameter space of the family $\widetilde{\mathcal{G}}= \{ \widetilde{g}_{a,b}(z)= z^3 + 3a^2z+b,\ a \geq 0, b\in \mathbb{R}\}$ of real cubic polynomials also contains Tricorn-like sets. However, the definition of Tricorn-like sets, the proof of umbilical cord wiggling, and discontinuity of straightening for the family $\widetilde{\mathcal{G}}$ are completely analogous to those for the family $\mathcal{G}$. Hence we work out the details only for the family $\mathcal{G}$.
\end{remark}

\subsection{The Hyperbolic Component of Period One}\label{sec_per_one}
Before studying renormalizations, we give an explicit description of the hyperbolic component of period one of $\mathcal{G}$.

Let,
\[
 \per_p(\lambda)=\{(a,b): \ g_{a,b}\mbox{ has a periodic point of period $p$ with multiplier $\lambda$}\}.
\]

It is easy to see that each $g_{a,b}$ in our family has exactly one real fixed point $x$, and exactly two non-real fixed points. 

\begin{lemma}[Indifferent Fixed Points]\label{indiff_fixed}
 If $g_{a,b}$ has an indifferent fixed point,
 then it is real and its multiplier is $-1$.
\end{lemma}

\begin{proof}
 First observe that there is no parabolic fixed point of
 multiplier one.
 In fact, fixed points are the roots of
 \[
 g_{a,b}(z)-z=-z^3-(3a^2+1)z+b.
 \]
 If a fixed point is parabolic of multiplier one, then the discriminant of $g_{a,b}(z)-z$ would vanish: i.e.,\
 \[
 27b^2=-4(3a^2+1)^3.
 \]
 Clearly, there is no real $(a,b)$ satisfying this equation, and hence $g_{a,b}$ cannot have a parabolic fixed point of multiplier $1$.

 Assume an indifferent fixed point $y$ is not real. Then there is also a symmetric (non-real) indifferent fixed point $\overline{y}$. Since we have already seen that the common multiplier of $y$ and $\overline{y}$ is not equal to $1$, the invariant external rays (i.e.,\ of angles $0$ and $1/2$) cannot land at $y$ and $\overline{y}$. Therefore, those rays must land at the other fixed point $x$ on the real axis.
 
 The critical points $\pm ai$ of $g_{a,b}$ are on the imaginary axis.
 Therefore, $g_{a,b}$ is monotone decreasing and
 $g_{a,b}^{\circ 2}$ is monotone increasing on $\R$.
 If $K(g_{a,b}) \cap \R$ contains an interval,
 then there must be a non-repelling fixed point for $g_{a,b}^{\circ 2}$ on $\mathbb{R}$. This is
 impossible because we already have two non-repelling cycles (in fact, fixed points).
 Therefore, we have $K(g_{a,b}) \cap \R=\{x\}$.
 By symmetry, the external rays at angles $1/4$ and $3/4$ are contained in 
 the real line (these rays have period $2$). Hence they both land at $x$.
 Therefore, $x$ is the landing point of periodic external rays of different
 periods, which is a contradiction.

 Therefore, any indifferent fixed point must be real, and since its multiplier is also real and not equal to one, 
 it is equal to $-1$.
\end{proof}

\begin{figure}[ht!]
\begin{minipage}{0.5\linewidth}
\begin{center}
\includegraphics[width=0.96\linewidth]{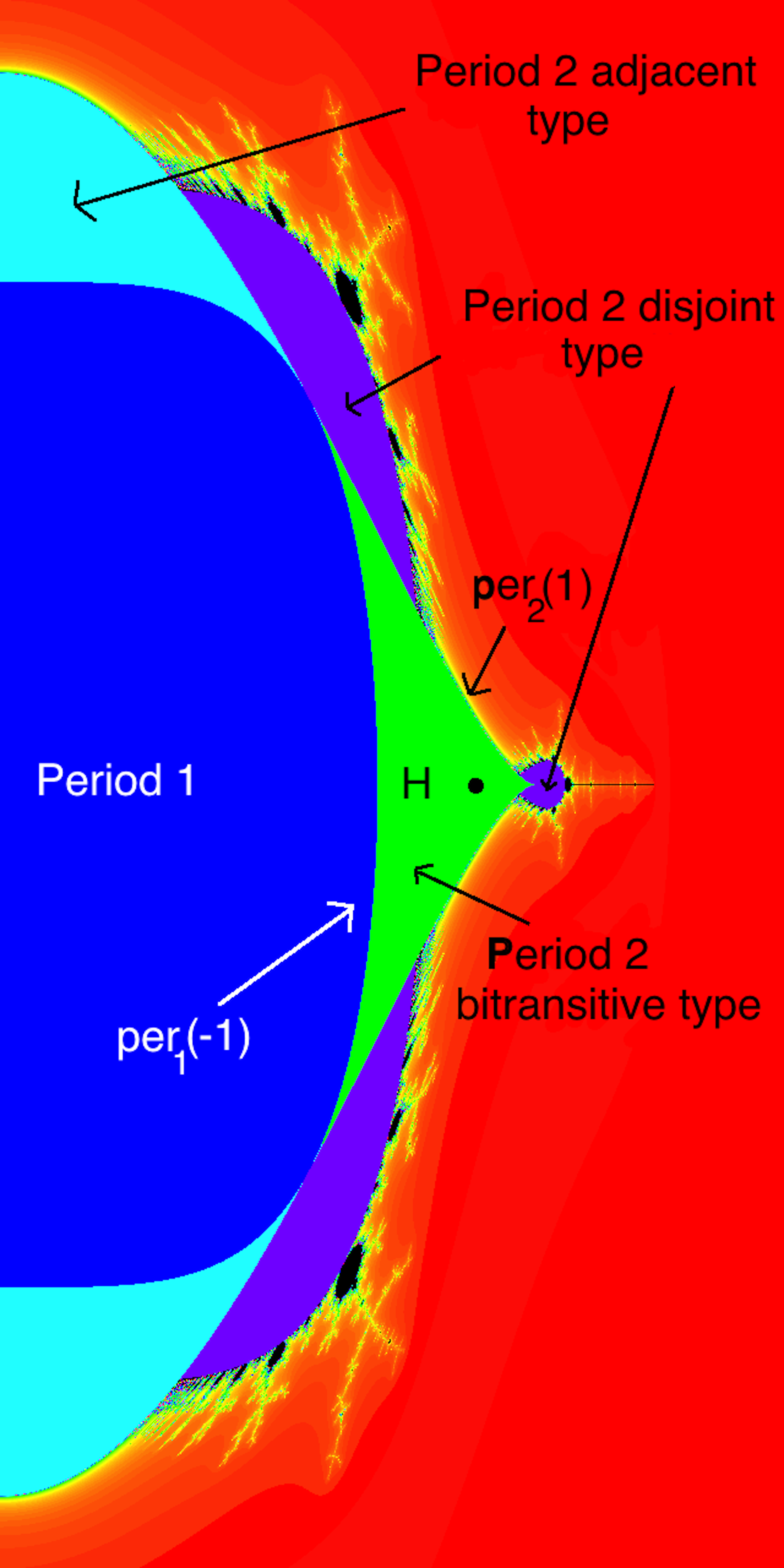}
\end{center}
\end{minipage}
\begin{minipage}{0.49\linewidth}
\includegraphics[width=0.96\linewidth]{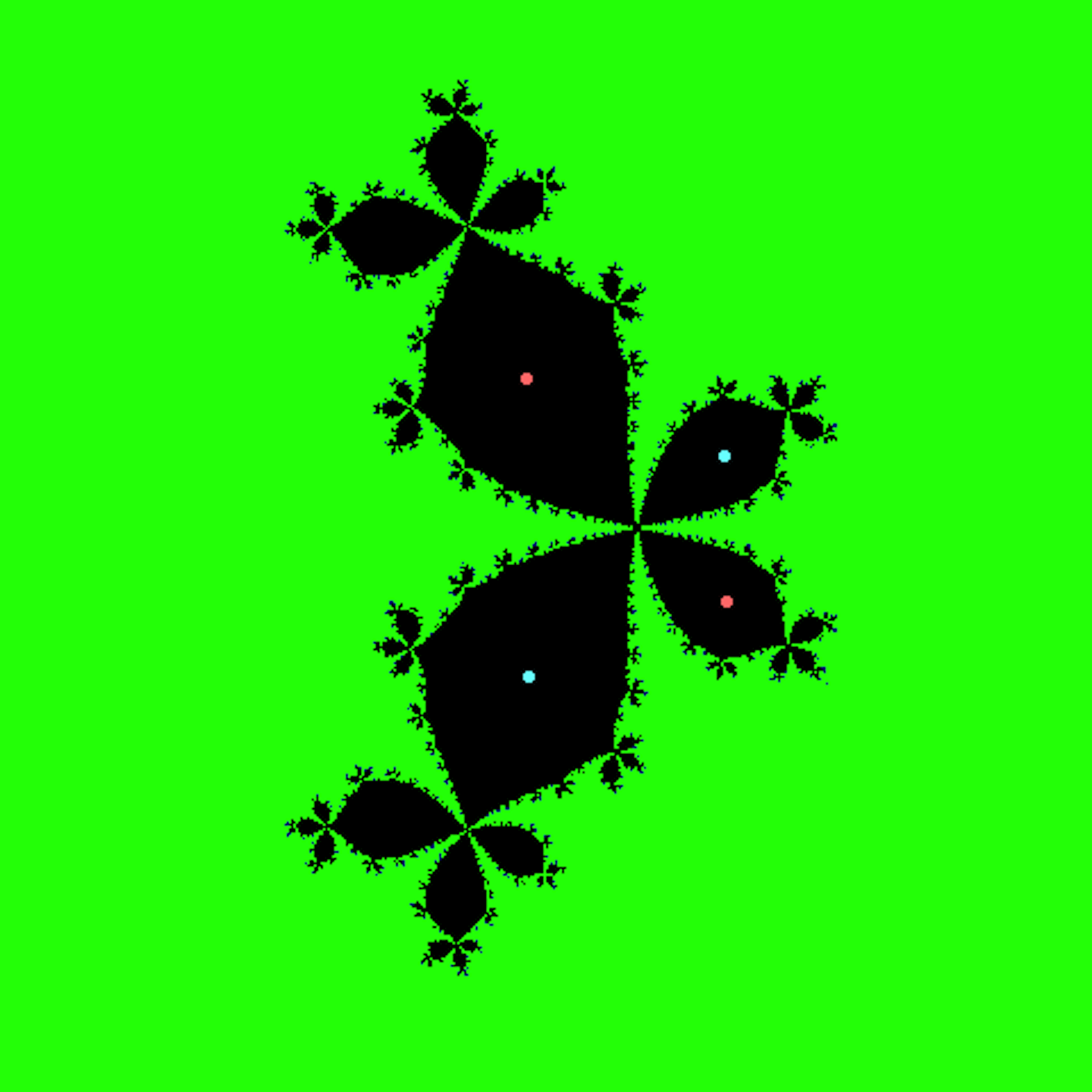}
\includegraphics[width=0.96\linewidth]{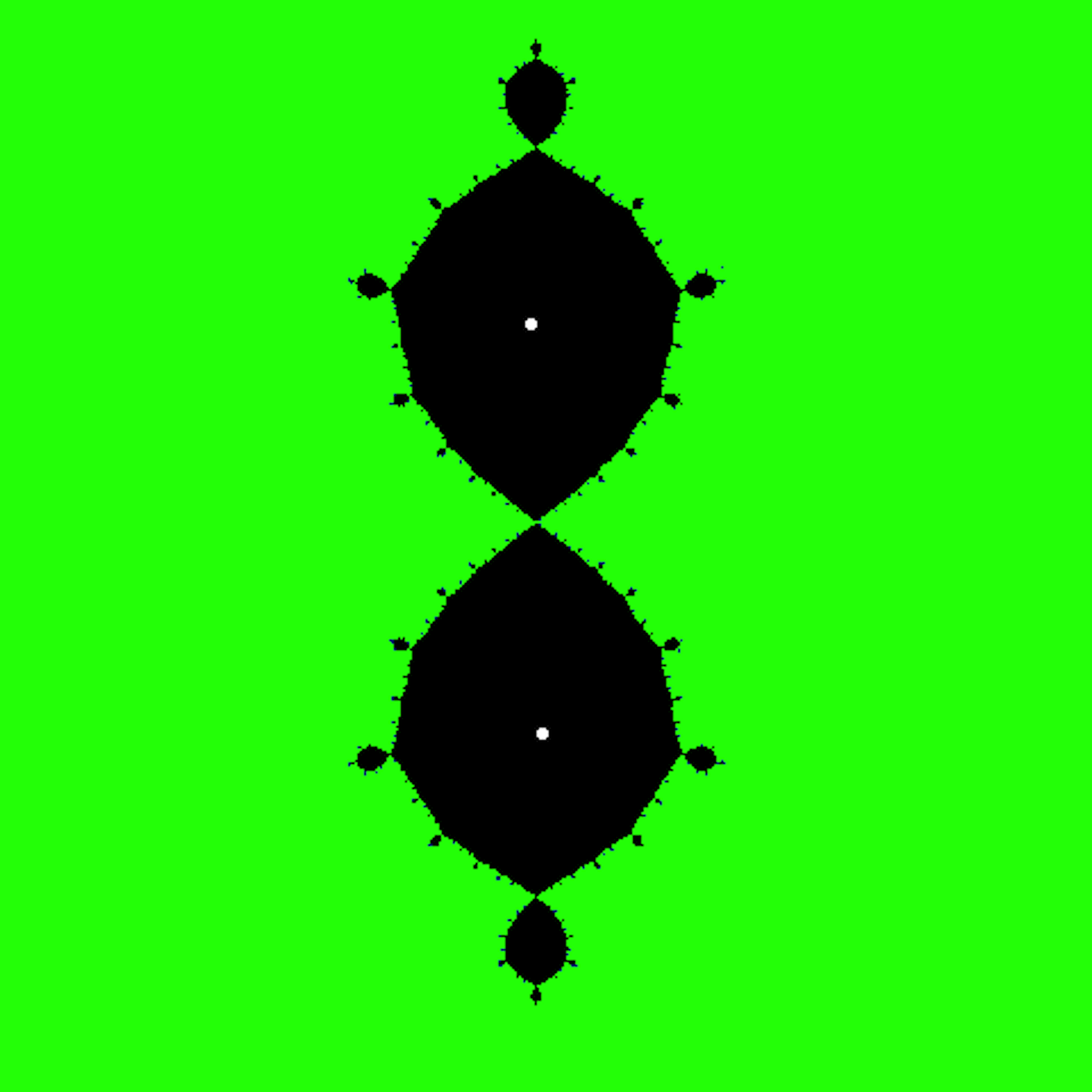}
\end{minipage}
\caption{Left: A cartoon of the parameter space of the family $\mathcal{G}$ highlighting the period $1$ and period $2$ components. For $(a_0,b_0)=(\frac{1}{\sqrt{2}},0)$, the corresponding renormalization locus $\mathcal{R}(a_0, b_0)$ fails to be compact precisely along a sub-arc of $\per_1(-1)\cup\per_2(1)$. Top right: The filled Julia set of the center of a disjoint-type hyperbolic component of period $2$ with the two distinct super-attracting $2$-cycles marked. Bottom right: The filled Julia set of the center $(\frac{1}{\sqrt{2}},0)$ of the unique bitransitive hyperbolic component of period $2$ with its super-attracting $2$-cycle marked. The two periodic bounded Fatou components touch at the origin, so $g_{\frac{1}{\sqrt{2}},0}$ is not primitive.}
\label{cubic_para}
\end{figure}

Therefore the set of parameters with indifferent fixed points
is equal to $\per_1(-1)$:
\begin{equation}
\begin{split}
 \per_1(-1) &= \{(a,b)\in \mathbb{R}_{\geq 0}\times \mathbb{R}: \textrm{Res}_z\left(g_{a,b}(z)-z,\ g_{a,b}'(z)+1\right)=0\} \\
 &= \{(a,b)\in \mathbb{R}_{\geq 0}\times \mathbb{R}: 4(3a^2-1)(3a^2+2)^2 + 27b^2=0\}. 
 \end{split}
  \label{eqn:per_1(-1)}
\end{equation}
Here, $\textrm{Res}_z(P(z),Q(z))$ stands for the resultant (of the two polynomials $P(z)$ and $Q(z)$), which is a polynomial in the coefficients of $P$ and $Q$ that vanishes if and only if $P$ and $Q$ have a common root.

Let $\Phi_1(a)=-4(3a^2-1)(3a^2+2)^2$.
Therefore, the (unique) hyperbolic component with attracting fixed point
is defined by
\[
 \cH_1 = \{(a,b)\in \mathbb{R}^2 :\ a \in [0,\textstyle{\frac{1}{\sqrt{3}}}),\ 27b^2 < \Phi_1(a)\}.
\]

We end this subsection by noting that there is only one bitransitive hyperbolic component of period $2$ in the family $\mathcal{G}$.

\begin{proposition}\label{bitran_center}
There is exactly one bitransitive hyperbolic component of period $2$ in the family $\mathcal{G}$, and the center of this component is $(1/\sqrt{2},0)$.
\end{proposition}
\begin{proof}
Let $(a_0,b_0)$ be the center of a bitransitive hyperbolic component of period $2$ in the family $\mathcal{G}$. Since $g_{a_0,b_0}$ must have two distinct critical points, we have that $a_0>0$. Moreover, it follows from our assumption that $$g_{a_0,b_0}(ia_0)=-ia_0\ \implies\ b_0+i(a_0-2a_0^3)=0.$$ Since $b_0$ is real, and $a_0$ is positive, it now follows that $(a_0,b_0)=(1/\sqrt{2},0)$. Since each hyperbolic component has a center, the result now follows.
\end{proof}

\begin{remark}
The unique bitransitive hyperbolic component of period $2$ touches the unique period $1$ hyperbolic component along a part of $\per_1(-1)$. There are three disjoint-type hyperbolic components of period $2$ that `bifurcate' from the unique period $2$ bitransitive hyperbolic component across sub-arcs of $\per_2(1)$. Furthermore, there are two adjacent-type hyperbolic components each of which touches the unique period $1$ hyperbolic component along a sub-arc of $\per_1(-1)$, and a disjoint-type hyperbolic component of period $2$ along a sub-arc of $\per_2(1)$ (see Figure~\ref{cubic_para}). The proofs of these facts are highly algebraic, and we omit them here. 
\end{remark}

\subsection{Centers of Bitransitive Components}\label{sec_bitran}
Since we will be concerned with renormalizations based at bitransitive hyperbolic components of the family $\mathcal{G}$, we need to take a closer look at the dynamics of the centers of such components. Throughout this subsection, we assume that $(a_0,b_0)$ is the center of a bitransitive hyperbolic component of period $2n$ (necessarily even due to the symmetry with respect to the real line); i.e.,\ $g_{a_0,b_0}^{\circ 2n}(\pm ia_0)=\pm ia_0$. Let $k$ be the smallest positive integer such that $g_{a_0,b_0}^{\circ k}(ia_0)=-ia_0$. Then $\iota\circ g_{a_0,b_0}^{\circ k}(ia_0)=ia_0$; i.e.,\ $g_{a_0,b_0}^{\circ k}\circ\iota(ia_0)=ia_0$, and $g_{a_0,b_0}^{\circ k}(-ia_0)=ia_0$. Therefore, $g_{a_0,b_0}^{\circ 2k}(ia_0)=ia_0$. This implies that $k=n$.

\begin{lemma}\label{real_fixed}
$K(g_{a_0,b_0})$ intersects $\mathbb{R}$  at a single point, which is the unique real fixed point $x$ of $g_{a_0,b_0}$.
\end{lemma}
\begin{proof}
Since $K(g_{a_0,b_0})$ is connected, full, compact, and symmetric with respect to the real line, its intersection with the real line is either a singleton $\{x\}$, or an interval $\left[p,q\right]$ with $q>p$. We assume the latter case. Then $p$ and $q$ are the landing points of the dynamical rays of $g_{a_0,b_0}$ at angles $1/4$ and $3/4$. So $\{p,q\}$ is a repelling $2$-cycle (cannot be parabolic as both critical points are periodic). Since $g_{a_0,b_0}$ is a real polynomial, and has no critical point on $\left[p,q\right]$, $g_{a_0,b_0}:\left[p,q\right]\rightarrow\left[p,q\right]$ is a strictly monotone map. But $g_{a_0,b_0}'(z)=-3(z^2+a_0^2)$, which is negative for $z$ in $\mathbb{R}$. Thus $g_{a_0,b_0}$ is strictly decreasing, and $g_{a_0,b_0}^{\circ 2}$ is strictly increasing on $\{p,q\}$. As $\{p,q\}$ is a repelling $2$-cycle, it follows that $\left[p,q\right]$ contains a non-repelling cycle of $g_{a_0,b_0}$. This is impossible because both critical points $\pm ia_0$ are periodic, and away from the real line. Hence, $\mathbb{R}\cap K(g_{a_0,b_0})= \{x\}$.
\end{proof}

\begin{definition}\label{primitive_def}
We say that a polynomial (respectively, anti-polynomial) $P$ with connected Julia set is \emph{primitive} if, for all distinct and bounded Fatou components $U$ and $V$ of $P$, we have that $\overline{U}\cap\overline{V}=\emptyset$.
\end{definition}

According to \cite[Theorem~D]{IK}, primitivity of the center of a hyperbolic component plays a key role in the study of the corresponding straightening map (in fact, it is shown there that primitivity is equivalent to the domain of the corresponding straightening map being non-empty and compact). It will thus be useful to know when the critically periodic polynomial $g_{a_0,b_0}$ is primitive. We answer this question in the following two lemmas.

Let us first discuss the special case when $n=1$. By Lemma~\ref{bitran_center}, the parameter $(a_0,b_0)=(1/\sqrt{2},0)$ is the center of the unique period $2$ bitransitive hyperbolic component of $\mathcal{G}$. More precisely, $g_{a_0,b_0}(ia_0)=-ia_0$, and $g_{a_0,b_0}(-ia_0)=ia_0$. Let $U_1$ and $U_2$ be the Fatou components of $g_{a_0,b_0}$ containing $ia_0$ and $-ia_0$ respectively. 

\begin{lemma}[The $n=1$ Case]\label{one_not_primitive}
$\partial U_1\cap\partial U_2=\{0\}$. In particular, $g_{\frac{1}{\sqrt{2}},0}$ is not primitive.
\end{lemma}
\begin{proof}

Observe that $g_{a_0,b_0}$ commutes with the reflection with respect to $i\mathbb{R}$, hence $-\iota\left(K(g_{a_0,b_0})\right)=K(g_{a_0,b_0})$. Since $\pm ia_0\in K(g_{a_0,b_0})$, and $K(g_{a_0,b_0})$ is connected, full, compact, it follows that $\left[-ia_0,ia_0\right]\subset K(g_{a_0,b_0})$. Since $0$ is the unique fixed point of $g_{a_0,b_0}$ on the real line, it follows by Lemma~\ref{real_fixed} that $\mathbb{R}\cap K(g_{a_0,b_0})= \{0\}$. Since $0$ is repelling, it belongs to the Julia set.

$g_{a_0,b_0}$ has no critical point in $(-ia_0,ia_0)$, so $g_{a_0,b_0}$ is strictly monotone there. Since $g_{a_0,b_0}'(0)=-3/2<0$, $g_{a_0,b_0}|_{\left[-ia_0,ia_0\right]}$ is strictly decreasing and $g_{a_0,b_0}^{\circ 2}|_{\left[-ia_0,ia_0\right]}$ is strictly increasing. It is now an easy exercise in interval dynamics to see that the $g_{a_0,b_0}^{\circ 2}$-orbit of each point in $\left(0,ia_0\right)$ converges to the super-attracting point $ia_0$, and hence $\left(0,ia_0\right]\subset U_1$. Since $0$ is in the Julia set, it follows that $0\in \partial U_1$.

A similar argument shows that $0\in \partial U_2$. But the boundaries of two bounded Fatou components of a polynomial cannot intersect at more than one point. This proves that $\partial U_1\cap\partial U_2=\{0\}$ (compare bottom right of Figure~\ref{cubic_para}).
\end{proof}

Now let $n>1$. 

\begin{lemma}[The $n>1$ Case]\label{not_one_primitive}
If $n$ is larger than $1$, then $g_{a_0,b_0}$ is primitive.
\end{lemma}
\begin{proof}
Using a Hubbard tree argument (compare \cite[Lemma~3.4]{NS}), it is easy to see that every periodic bounded Fatou component of $g_{a_0,b_0}$ has exactly one boundary point that is fixed by $g_{a_0,b_0}^{\circ 2n}$ and that is a cut-point of the Julia set. We call this point the \emph{root} of the periodic (bounded) Fatou component. Let $U_1$ and $U_2$ be two distinct Fatou components of $g_{a_0,b_0}$ with $\partial U_1\cap\partial U_2\neq \emptyset$. We can take iterated forward images to assume that $U_1$ and $U_2$ are periodic. Then the intersection $\partial U_1\cap \partial U_2$ consists only of the unique common root $x$ of $U_1$ and $U_2$.
 
Let $U_1,\ U_2,\ \cdots,\ U_r$ be all the periodic components touching at $x$. Since $g_{a_0,b_0}$ commutes with $\iota$ and there is only one cycle of periodic components, it follows that $g_{a_0,b_0}^{\circ n}(U_j)=\iota(U_j)$, for $j=1,\ 2,\ \cdots,\ r$. Moreover, $g_{a_0,b_0}^{\circ n}$ is a local orientation-preserving diffeomorphism from a neighborhood of $x$ to a neighborhood of $\overline{x}$. But if $r\geq 3$, it would reverse the cyclic order of the Fatou components $U_j$ touching at $x$. Hence, $r\leq 2$; i.e.,\ at most $2$ periodic (bounded) Fatou components can touch at $x$. This implies that $x$ has period $n$, and $g_{a_0,b_0}^{\circ n}(U_1)=U_2$. Hence $U_2=\iota(U_1)$. The upshot of this is that $\partial U_1$ and $\iota(\partial U_1)$ intersect at $x$. But by Lemma~\ref{real_fixed}, $x$ must be the unique real fixed point of $g_{a_0,b_0}$. This contradicts the assumption that $n>1$.

Therefore, all bounded Fatou components of $g_{a_0,b_0}$ have disjoint closures, and $g_{a_0,b_0}$ is primitive.
\end{proof}

\subsection{Renormalizations of Bitransitive Components, and Tricorn-like Sets}\label{proof_main_thm}

Let $(a_0,b_0)$ be the center of a bitransitive hyperbolic component $H$ of period $2n$; i.e.,\ $g_{a_0,b_0}^{\circ n}(ia_0) = -ia_0$ and $g_{a_0,b_0}^{\circ n}(-ia_0) = ia_0$ for some $n\geq 1$. Then there exists a neighborhood $U_0$ of the closure of the Fatou component containing $ia_0$ such that $U_0$ is compactly contained in $\iota \circ g_{a_0,b_0}^{\circ n}(U_0)$ with $\iota \circ g_{a_0,b_0}^{\circ n}:U_0\rightarrow (\iota \circ g_{a_0,b_0}^{\circ n})(U_0)$ proper (compare Figure~\ref{real cubic}). Since $\left(\iota \circ g_{a_0,b_0}^{\circ n}\right)\vert_{U_0}$ is an anti-holomorphic map of degree $2$, we have an anti-polynomial-like map of degree $2$ (with a connected filled Julia set) defined on $U_0$. The Straightening Theorem now yields a quadratic anti-holomorphic polynomial (with a connected filled Julia set) that is hybrid equivalent to $(\iota \circ g_{a_0,b_0}^{\circ n})\vert_{U_0}$ \cite[Theorem~5.3]{IM4}. One can continue to perform this renormalization procedure as the real cubic polynomial $g_{a,b}$ moves in the parameter space, and this defines a map from a suitable region in the parameter plane of real cubic polynomials to the Tricorn. We now proceed to define this region.

\begin{definition}[Rational Lamination]
The rational lamination of a holomorphic or anti-holomorphic polynomial $P$ (with connected Julia set) is defined as an equivalence relation on $\mathbb{Q}/\mathbb{Z}$ such that $\theta_1 \sim \theta_2$ if and only if the dynamical rays $R(\theta_1)$ and $R(\theta_2)$ of $P$ land at a common point on the Julia set of $P$. 
\end{definition}

Let $\lambda_{a,b}$ be the rational lamination of $g_{a,b}$. Define the \emph{combinatorial renormalization locus} to be
\begin{align*}
\mathcal{C}(a_0,b_0) &= \{ (a,b)\in \mathbb{R}^2: a\geq 0, \lambda_{a,b} \supset \lambda_{a_0,b_0}\}.
\end{align*}

Since a rational lamination is an equivalence relation on $\mathbb{Q}/\mathbb{Z}$, it is a subset of $\mathbb{Q}/\mathbb{Z} \times \mathbb{Q}/\mathbb{Z}$, and hence, subset inclusion makes sense. By definition, for $(a,b) \in \mathcal{C}(a_0,b_0)$, the external rays at $\lambda_{a_0,b_0}$-equivalent angles for $g_{a,b}$ land at the same point. Hence those rays divide the filled Julia set of $g_{a,b}$ into `\emph{fibers}'. Let $K$ be the fiber containing the critical point $ai$. Then $(\iota\circ g_{a,b}^{\circ n})(K) = K$. We say $g_{a,b}$ is $(a_0,b_0)$-\emph{renormalizable} if there exists an anti-holomorphic quadratic-like restriction $\iota\circ g_{a,b}^{\circ n} : U^{\prime}_{a,b} \rightarrow U_{a,b}$ with filled Julia set equal to $K$. Let the renormalization locus $\mathcal{R}(a_0,b_0)$ with combinatorics $\lambda_{a_0,b_0}$ be:
\begin{center}
$\mathcal{R}(a_0,b_0) = \lbrace (a,b) \in \mathcal{C}(a_0,b_0) : g_{a,b} \hspace{1mm}$ is $(a_0,b_0)$-renormalizable$\rbrace$.
\end{center}

Using \cite[Theorem~5.3]{IM4}, for each $(a,b)\in \mathcal{R}(a_0,b_0)$, we can straighten $\iota\circ g_{a,b}^{\circ n}: U'_{a,b}\rightarrow U_{a,b}$ to obtain a quadratic anti-holomorphic polynomial $f_c$. This defines the straightening map

\begin{center}
 $\begin{array}{rccc}
  \chi_{a_0,b_0}:& \mathcal{R}(a_0,b_0)& \to & \mathcal{M}_2^* \\
   & (a,b) & \mapsto  & c.
  \end{array}$\\
\end{center} 
The proof of the fact that our definition of $\chi_{a_0,b_0}$ agrees with the general definition of straightening maps \cite{IK} goes as in \cite[\S 5]{IM4}. It then follows from \cite[Theorem~B]{IK} that $\chi_{a_0,b_0}: \mathcal{R}(a_0,b_0) \to \mathcal{M}_2^*$ is injective. 

In order to discuss compactness of the renormalization locus $\mathcal{R}(a_0,b_0)$, we have to distinguish between the cases $n=1$ and $n>1$. By Lemma~\ref{not_one_primitive}, $g_{a_0,b_0}$ is primitive whenever $n>1$. Therefore \cite[Theorem~D]{IK} implies that for $n>1$, the renormalization locus $\cR(a_0,b_0)$ is compact, and coincides with the combinatorial renormalization locus $\mathcal{C}(a_0,b_0)$. On the other hand, when $n=1$, Lemma~\ref{one_not_primitive} tells that $g_{a_0,b_0}$ is not primitive. Hence according to \cite[Theorem~D]{IK}, $\cR(a_0,b_0)$ is a proper non-compact subset of $\cC(a_0,b_0)$ for $n=1$. 

Regarding the image of $\cR(a_0,b_0)$, we have the following result for all $n\geq1$.

\begin{proposition}
 \label{prop:real cubic onto hyp}
The image of the straightening map $\chi_{a_0,b_0}$ contains the hyperbolicity locus in $\Int \cM_2^*$.
\end{proposition}

\begin{remark}
1) Indeed, the proposition holds for any straightening map $\chi_{a_0,b_0}$ for anti-holomorphic renormalizations in the family $\{g_{a,b}\}$.

2) In the proof below, we parametrize cubic polynomials to be monic so that we can apply the results of \cite{IK} on general straightening maps.
\end{remark}

\begin{proof}
We complexify the family and consider the straightening map defined there.

Let $\Poly(3)=\{P_{A,b}(z)=z^3-3Az+b;\ (A,b) \in \C^2\}$ denote the complex cubic family. Observe that $h_{a,b}=P_{a^2,bi}$. Let $\cC(2\times 2)$ be the connectedness locus of the biquadratic family $\{(z^2+a)^2+b:a,b\in\mathbb{C}\}$, and $\widetilde{\chi}_{A_0,b_0i}:\widetilde{\cC}(A_0,b_0i) \to \cC(2\times 2)$ be the straightening map for $(A_0,b_0i)$-renormalization in $\Poly(3)$, where $A_0=a_0^2$.

Applying \cite[Theorem~C]{IK} to the current setting, we conclude that for any hyperbolic parameter $c \in \cM_2^*$, there exists some $(A,b) \in \C^2$ such that $\widetilde{\chi}_{A_0,b_0i}(A,b)$ is defined and equal to the biquadratic polynomial $(z^2+\overline{c})^2+c$, which is the second iterate of $\overline{z}^2+c$. Hence if $A$ is positive real and $b$ is purely imaginary, then it follows that $\chi_{a_0,b_0}(a,\frac{b}{i})=c$, where $a=\sqrt{A}$.
 
In fact, let us consider $P(z)=P_{A,b}(z)=z^3-3Az+b$. Since $h_{a_0,b_0}=P_{a_0^2,b_0i}$ is symmetric with respect to the imaginary axis, we have that $\varphi(P_{A,b}(\varphi(z)))=P_{\bar{A},-\bar{b}}$ is also $(A_0,b_0i)$-renormalizable, where $\varphi(z)=-\bar{z}$ is the reflection with respect to the imaginary axis. Moreover, since $\varphi$ exchanges the critical points $\pm a_0$ for $P_{A_0,b_0i}$, we have 
 \[
  \widetilde{\chi}_{A_0,b_0i}(\bar{A},-\bar{b})
 = (z^2+\overline{c})^2+c = \widetilde{\chi}_{A_0,b_0i}(A,b).
 \]
Therefore, by \cite[Theorem~B]{IK}, $(\bar{A},-\bar{b})=(A,b)$, equivalently, $A \in \R$  and $b \in i\R$.
\end{proof}

With these preliminary results at our disposal, we can now set up the foundation for the key technical theorem (of this section) to the effect that all `umbilical cords' away from the line $\{b=0\}$ `wiggle'. 

Let $H_1$, $H_2$, $H_3$ be the hyperbolic components of period $3$ of $\mathcal{M}_2^*$ (by \cite[Theorem~1.3]{MNS}, there are exactly $3$ of them). By Proposition~\ref{prop:real cubic onto hyp}, we have that 
$$
\chi_{a_0,b_0}(\mathcal{R}(a_0,b_0))\supset H_1\cup H_2\cup H_3.
$$ 
We claim that at most one of the hyperbolic components $\chi_{a_0,b_0}^{-1} \left( H_i\right)\subset\mathcal{R}(a_0,b_0)$ ($i\in\{1,2,3\}$) can intersect the line $\{b=0\}$. To see this, first observe that the anti-holomorphic involution $z\mapsto -z^*$ conjugates $g_{a,b}$ to $g_{a,-b}$. Hence, if a hyperbolic component of $\mathcal{G}$ intersects the line $\{b=0\}$, then it must be symmetric with respect to this line, and the center of this hyperbolic component must lie on $\{b=0\}$. Let us first suppose that $b_0=0$; i.e.,\ the hyperbolic component $H$ (which has $(a_0,b_0)$ as its center) is real-symmetric. In this case, reflection with respect to the line $\{b=0\}$ preserves $\cup_{i=1}^3\chi_{a_0,b_0}^{-1} \left( H_i\right)$. It follows that reflection with respect to the line $\{b=0\}$ must fix precisely one and interchange the other two hyperbolic components among $\chi_{a_0,b_0}^{-1} \left( H_i\right), i\in\{1,2,3\}$. Hence, exactly one of the hyperbolic components $\chi_{a_0,b_0}^{-1} \left( H_i\right), i\in\{1,2,3\},$ intersects the line $\{b=0\}$. Now suppose that $b_0\neq0$ (so $H$ lies off the line $\{b=0\}$), and the hyperbolic component $\chi_{a_0,b_0}^{-1} \left( H_i\right)$ (for some $i\in\{1,2,3\}$) intersects $\{b=0\}$. Since $b_0\neq 0$, the reflected hyperbolic component $\iota(H)$, which has its center at $(a_0,-b_0)$, is disjoint from $H$. Thus, the assumption that the hyperbolic component $\chi_{a_0,b_0}^{-1} \left( H_i\right)$ (for some $i\in\{1,2,3\}$) intersects $\{b=0\}$ implies that $\chi_{a_0,b_0}^{-1} \left( H_i\right)\subset\mathcal{C}(a_0,b_0)$ would intersect and hence coincide with the hyperbolic component $\chi_{a_0,-b_0}^{-1} \left( H_j\right)\subset\mathcal{C}(a_0,-b_0)$ (for some $j\in\{1,2,3\}$). But this is impossible as the centers of these components belong to two different combinatorial renormalization loci (namely, $\mathcal{C}(a_0,b_0)$ and $\mathcal{C}(a_0,-b_0)$), and thus must have different rational laminations. We conclude that in either case, at most one of the hyperbolic components $\chi_{a_0,b_0}^{-1} \left( H_i\right)$ ($i\in\{1,2,3\}$) can intersect the line $\{b=0\}$. 

Therefore, we can and will pick $i$ such that $\chi_{a_0,b_0}^{-1} \left( H_i\right)$ does not intersect $\{b=0\}$, and set $$H':=\chi_{a_0,b_0}^{-1} \left( H_i\right).$$

By construction of $H'$, each map $g_{a,b}$ in $H'$ has an attracting cycle of period $6n$ such that every point in this cycle is fixed by the anti-holomorphic map $\iota\circ g_{a,b}^{\circ 3n}$. Hence, it follows from the arguments of \cite[Lemma~2.8]{MNS} that each $g_{a,b}$ on $\partial H'$ has a parabolic cycle of period $6n$ such that every point in this cycle is fixed by $\iota\circ g_{a,b}^{\circ 3n}$. Moreover, since $g_{a,b}$ has two critical points, the same result also implies that there can be at most two attracting petals at each such parabolic point. We will call $(a,b)\in\partial H'$ a \emph{simple parabolic parameter} (respectively, a \emph{parabolic cusp}) if the number of attracting petals at each parabolic point is $1$ (respectively, $2$). 

Now, for a simple parabolic parameter $(a,b)\in\partial H'$, each attracting petal is fixed under $\iota\circ g_{a,b}^{\circ 3n}$, and hence \cite[Lemma~3.1]{MNS} provides us with an attracting Fatou coordinate (unique up to a real additive constant) that conjugate the map $\iota\circ g_{a,b}^{\circ 3n}$ (on the petal) to the map $\zeta\mapsto\overline{\zeta}+1/2$ (on a right half-plane). We will call the imaginary part of $(\iota\circ g_{a,b}^{\circ 3n})(ai)$ in the above Fatou coordinate the \emph{critical {\'E}calle height} of $g_{a,b}$. The pre-image of the real line under this attracting Fatou coordinate will be referred to as the \emph{attracting equator}. Clearly, the same construction can be carried out in the repelling petal as well giving rise to a repelling Fatou coordinate that conjugate $\iota\circ g_{a,b}^{\circ 3n}$ (on the repelling petal) to the map $\zeta\mapsto\overline{\zeta}+1/2$ (on a left half-plane). The pre-image of the real line under this repelling Fatou coordinate will be called the \emph{repelling equator}. 

The proof of \cite[Theorem~1.2]{MNS} (which proves a structure theorem for the boundaries of odd period hyperbolic components of the multicorns) can now be adapted to the current setting to show that $\partial H'$ consists of $3$ parabolic cusps and $3$ parabolic arcs which are parametrized by the critical {\'E}calle height (also see \cite[Theorem~3.2]{MNS}). 
 
For our purposes, the most important parameter on a parabolic arc is the parameter with critical {\'E}calle height $0$. Let $(\widetilde{a},\widetilde{b})$ be the critical {\'E}calle height $0$ parameter on the root arc (such that the unique parabolic cycle disconnects the Julia set) of $\partial H'$. We denote the unique Fatou component of $g_{\widetilde{a},\widetilde{b}}$ containing the critical point $\widetilde{a}i$ by $U$, and the unique parabolic periodic point of $g_{\widetilde{a},\widetilde{b}}$ on $\partial U$ by $\widetilde{z}$. Since $g_{\widetilde{a},\widetilde{b}}$ commutes with $\iota$, $\iota(U)$ is the unique Fatou component of $g_{\widetilde{a},\widetilde{b}}$ containing the critical point $-\widetilde{a}i$, and $\widetilde{z}^*$ is the unique parabolic periodic point of $g_{\widetilde{a},\widetilde{b}}$ on $\partial \iota(U)$. 
 
We define a \emph{loose parabolic tree} of $g_{\widetilde{a},\widetilde{b}}$ as a minimal tree within the (path connected) filled Julia set that connects the parabolic orbit (of period $6n$) and the critical orbits such that it intersects the closure of any bounded Fatou component at no more than two points. Since the filled Julia set of a polynomial is full, any loose parabolic tree is uniquely defined up to homotopies within bounded Fatou components. It is easy to see that any loose parabolic tree intersects the Julia set in a Cantor set, and these points of intersection are the same for any loose tree (note that the parabolic cycle of $g_{\widetilde{a},\widetilde{b}}$ is simple, and hence any two periodic Fatou components have disjoint closures). By construction, the forward image of a loose parabolic tree is contained in a loose parabolic tree (see \cite[\S 2.3]{IM4} for more details). 

We will now show that if the `umbilical cord' of $H'$ lands, then the two restricted maps $g_{\widetilde{a},\widetilde{b}}^{\circ 6n}\vert_{B(\widetilde{z},\epsilon)}$ and $g_{\widetilde{a},\widetilde{b}}^{\circ 6n}\vert_{B(\widetilde{z}^*,\epsilon)}$ (where $B(w,\epsilon)$ denotes the round disk of radius $\epsilon$ centered at $w$) are conformally conjugate.

\begin{lemma}\label{umbilical_lands_germs_conjugate}
If there exists a path $\gamma:\left[0,\delta\right]\to \mathbb{R}^2$ such that $\gamma(0)=(\widetilde{a},\widetilde{b})$, and $\gamma(\left(0,\delta\right]) \subset \mathcal{C}(\mathcal{G})\setminus \overline{H'}$, then the two restricted maps $g_{\widetilde{a},\widetilde{b}}^{\circ 6n}\vert_{B(\widetilde{z},\epsilon)}$ and $g_{\widetilde{a},\widetilde{b}}^{\circ 6n}\vert_{B(\widetilde{z}^*,\epsilon)}$ are conformally conjugate by a local biholomorphism $\eta$ such that $\eta$ maps $g_{\widetilde{a},\widetilde{b}}^{\circ 3rn}(-\widetilde{a}i)$ to $g_{\widetilde{a},\widetilde{b}}^{\circ 3rn}(\widetilde{a}i)$, for $r$ large enough.
\end{lemma}
\begin{proof}
The proof follows the strategy of \cite[\S 3]{IM4}, where the corresponding result was proved for quadratic anti-polynomials. Since we deal with a different family in this paper, we include the details for completeness.

Since any two bounded Fatou components of $g_{\widetilde{a},\widetilde{b}}$ have disjoint closures, it follows that any parabolic tree must traverse infinitely many bounded Fatou components, and intersect their boundaries at pre-parabolic points. Furthermore, any loose parabolic tree intersects the Julia set at a Cantor set of points.

We first claim that the repelling equator at $\widetilde{z}$ is contained in a loose parabolic tree of $g_{\widetilde{a},\widetilde{b}}$. To this end, it suffices to show that the repelling equator is contained in the filled Julia set of $g_{\widetilde{a},\widetilde{b}}$. If this were not true, then there would exist dynamical rays (in the dynamical plane of $g_{\widetilde{a},\widetilde{b}}$) crossing the equator and traversing an interval of outgoing {\'E}calle heights $\left[-x, x\right]$ with $x>0$. Since dynamical rays and Fatou coordinates depend continuously on the parameter, this would remain true after perturbation. For $s > 0$, the critical orbits of $g_{\gamma(s)}$ ``transit'' from the incoming {\'E}calle cylinder to the outgoing cylinder (the two critical orbits are related by the conjugacy $\iota$); as $s \downarrow 0$, the image of the critical orbits in the outgoing {\'E}calle cylinder has (outgoing) {\'E}calle height tending to $0$, while the phase tends to $-\infty$ \cite[Lemma~2.5]{IM}. Therefore, there exists $s \in \left(0, \delta\right)$ arbitrarily close to $0$ for which the critical orbit(s), projected into the incoming cylinder, and sent by the transit map to the outgoing cylinder, land(s) on the projection of some dynamical ray that crosses the equator. But in the dynamics of $g_{\gamma(s)}$, this means that the critical orbits lie in the basin of infinity, i.e.,\ such a parameter $\gamma(s)$ lies outside $\mathcal{C}(\mathcal{G})$. This contradicts our assumption that $\gamma((0,\delta]) \subset \mathcal{C}(\mathcal{G}) \setminus \overline{H'}$, and completes the proof of the claim.

We now pick any bounded Fatou component $U'\neq U$ that the repelling equator hits. Then the equator intersects $\partial U'$ at some pre-parabolic point $z'$. Consider a small piece $\Gamma'$ of the equator with $z'$ in its interior. Since $z'$ eventually falls on the parabolic orbit, some large iterate of $g_{\widetilde{a},\widetilde{b}}$ maps $z'$ to $\widetilde{z}$ by a local biholomorphism carrying $\Gamma'$ to an analytic arc $\Gamma$ (say, $\Gamma = g_{\widetilde{a},\widetilde{b}}^{\circ l} (\Gamma')$ for some $l$ large) passing through $\widetilde{z}$. We will show that $\Gamma$ agrees with the repelling equator (up to truncation). Indeed, the repelling equator, and the curve $\Gamma$ are both parts of two loose parabolic trees (recall that any forward iterate of a loose parabolic tree is contained in a loose parabolic tree), and hence must coincide along a Cantor set of points on the Julia set. As analytic arcs, they must thus coincide up to truncation. In particular, the part of $\Gamma$ not contained in $U$ is contained in the repelling equator, and is forward invariant. We can straighten the analytic arc $\Gamma$ to an interval $(-\delta', \delta') \subset \mathbb{R}$ by a local biholomorphism $\alpha : V \rightarrow \mathbb{C}$ such that $\widetilde{z} \in V$, and $\alpha(\widetilde{z})=0$ (for convenience, we choose $V$ such that it is symmetric with respect to $\Gamma$). This local biholomorphism conjugates the parabolic germ of $g_{\widetilde{a},\widetilde{b}}^{\circ 6n}$ at $\widetilde{z}$ to a germ that fixes $0$. Moreover, the conjugated germ maps small enough positive reals to positive reals. Clearly, this must be a real germ; i.e.,\ its Taylor series has real coefficients. Thus, the parabolic germ of $g_{\widetilde{a},\widetilde{b}}^{\circ 6n}$ at $\widetilde{z}$ is conformally conjugate to a real germ.

Since $\iota$ commutes with $g_{\widetilde{a},\widetilde{b}}$, one can carry out the preceding construction at the parabolic point $\widetilde{z}^*$, and show that the parabolic germ of $g_{\widetilde{a},\widetilde{b}}^{\circ 6n}$ at $\widetilde{z}^*$ is also conformally conjugate to a real germ. In fact, the role of $\Gamma'$ is now played by $\iota(\Gamma')$, and hence, the role of $\Gamma$ is played by $g_{\widetilde{a},\widetilde{b}}^{\circ l} (\iota(\Gamma')) = \iota(\Gamma)$. Then the biholomorphism $\iota \circ \alpha \circ \iota : \iota(V) \rightarrow \mathbb{C}$ straightens $\iota(\Gamma)$. Conjugating the parabolic germ of $g_{\widetilde{a},\widetilde{b}}^{\circ 6n}$ at $\widetilde{z}^*$ by $\iota \circ \alpha \circ \iota$, one recovers the same real germ as in the previous paragraph. Thus, the parabolic germs given by the restrictions of $g_{\widetilde{a},\widetilde{b}}^{\circ 6n}$ in neighborhoods of $\widetilde{z}$ and $\widetilde{z}^*$ are conformally conjugate. 

Note that as $g_{\widetilde{a},\widetilde{b}}$ has critical {\'E}calle height $0$, its critical orbits (in $U$) lie on the attracting equator. Moreover, since the attracting equator at $\widetilde{z}$ is mapped to the real line by $\alpha$, the conjugacy $$\eta:=\left(\iota \circ \alpha \circ \iota\right)^{-1}\circ\alpha$$ preserves the critical orbits; i.e.,\ it maps $g_{\widetilde{a},\widetilde{b}}^{\circ 3rn}(-\widetilde{a}i)$ to $g_{\widetilde{a},\widetilde{b}}^{\circ 3rn}(\widetilde{a}i)$ (for $r$ large enough, so that $g_{\widetilde{a},\widetilde{b}}^{\circ 3rn}(-\widetilde{a}i)=(\iota\circ g_{\widetilde{a},\widetilde{b}}^{\circ 3rn})(\widetilde{a}i)$ is contained in the domain of definition of $\alpha$). The proof is now complete.
\end{proof}

\begin{remark}
It follows from the proof of the previous lemma that the real-analytic curve $\Gamma$ passing through $\widetilde{z}$ is invariant under $\iota\circ g_{\widetilde{a},\widetilde{b}}^{\circ 3n}$. Indeed, $\Gamma$ is formed by the parabolic point $\widetilde{z}$, and parts of the attracting and the repelling equator at $\widetilde{z}$. 
\end{remark}

The next lemma improves the conclusion of Lemma~\ref{umbilical_lands_germs_conjugate}, and shows that landing of the umbilical cord of $H'$ implies the existence of a conformal conjugacy between the polynomial-like restrictions of $g_{\widetilde{a},\widetilde{b}}^{\circ 6n}$ in some neighborhoods of $\overline{U}$ and $\overline{\iota(U)}$ (respectively).

\begin{lemma}\label{umbilical_cubic}
If there exists a path $\gamma:\left[0,\delta\right]\to \mathbb{R}^2$ such that $\gamma(0)=(\widetilde{a},\widetilde{b})$, and $\gamma(\left(0,\delta\right]) \subset \mathcal{C}(\mathcal{G})\setminus \overline{H'}$, then the two polynomial-like restrictions of $g_{\widetilde{a},\widetilde{b}}^{\circ 6n}$ in some neighborhoods of $\overline{U}$ and $\overline{\iota(U)}$ are conformally conjugate.
\end{lemma}

\begin{proof}
We will first show that the conformal conjugacy $\eta= \iota \circ \alpha^{-1} \circ \iota \circ \alpha : V \rightarrow \iota(V)$ (from Lemma~\ref{umbilical_lands_germs_conjugate}) between the restrictions of $g_{\widetilde{a},\widetilde{b}}^{\circ 6n}$ around $\widetilde{z}$ and $\widetilde{z}^*$ can be extended to all of $U$. To this end, let us choose the attracting Fatou coordinate $\psi^{\mathrm{att}}_{\widetilde{a},\widetilde{b}}$ at $\widetilde{z}$ normalized so that it maps the equator to the real line, and $\psi^{\mathrm{att}}_{\widetilde{a},\widetilde{b}}(\widetilde{a}i) = 0$. Then, $\psi^{\mathrm{att}}_{\widetilde{a},\widetilde{b}}$ conjugates the anti-holomorphic return map $\iota\circ g_{\widetilde{a},\widetilde{b}}^{\circ 3n}$ (on an attracting petal at $\widetilde{z}$) to $\zeta\mapsto\overline{\zeta}+1/2$ (on a right half-plane). This naturally determines a preferred attracting Fatou coordinate $(\iota \circ \psi^{\mathrm{att}}_{\widetilde{a},\widetilde{b}} \circ \iota)$ at $\widetilde{z}^*$ such that it conjugates the anti-holomorphic return map $\iota\circ g_{\widetilde{a},\widetilde{b}}^{\circ 3n}$ (on an attracting petal at $\widetilde{z}^*$) to $\zeta\mapsto\overline{\zeta}+1/2$, and $(\iota\circ\psi^{\mathrm{att}}_{\widetilde{a},\widetilde{b}}\circ\iota)(-\widetilde{a}i) = 0$.

Since $\eta$ is a conjugacy between parabolic germs, it maps some attracting petal (not necessarily containing $\widetilde{a}i$) $P\subset V$ at $\widetilde{z}$ to some attracting petal $\iota(P) \subset \iota(V)$ at $\widetilde{z}^*$. Hence, $\psi^{\mathrm{att}}_{\widetilde{a},\widetilde{b}} \circ \eta^{-1}$ is an attracting Fatou coordinate for $g_{\widetilde{a},\widetilde{b}}^{\circ 6n}$ at $\widetilde{z}^*$. By the uniqueness of Fatou coordinates, we have that 
$$(\psi^{\mathrm{att}}_{\widetilde{a},\widetilde{b}} \circ \eta^{-1})(z) = (\iota \circ \psi^{\mathrm{att}}_{\widetilde{a},\widetilde{b}} \circ \iota)(z) + \lambda,$$
for some $\lambda \in \mathbb{C}$, and for all $z$ in their common domain of definition. There is some large $l\in\N$ for which $g_{\widetilde{a},\widetilde{b}}^{\circ 6ln}(-\widetilde{a}i)$ belongs to $\iota(V)$, the domain of definition of $\eta^{-1}$. By definition, 
$$\psi^{\mathrm{att}}_{\widetilde{a},\widetilde{b}} \circ \eta^{-1}(g_{\widetilde{a},\widetilde{b}}^{\circ 6ln}(-\widetilde{a}i))= \psi^{\mathrm{att}}_{\widetilde{a},\widetilde{b}} \circ \alpha ^{-1} \circ \iota \circ \alpha \circ \iota \circ g_{\widetilde{a},\widetilde{b}}^{\circ 6ln}(-\widetilde{a}i)$$
$$\hspace{8mm} =\psi^{\mathrm{att}}_{\widetilde{a},\widetilde{b}} \circ \alpha ^{-1} \circ \iota \circ \alpha \circ g_{\widetilde{a},\widetilde{b}}^{\circ 6ln}\circ \iota (-\widetilde{a}i)= \psi^{\mathrm{att}}_{\widetilde{a},\widetilde{b}} \left(\alpha^{-1}\left( \iota\left(\alpha\left(g_{\widetilde{a},\widetilde{b}}^{\circ 6ln}(\widetilde{a}i)\right)\right)\right)\right)$$
$$\hspace{-1.5cm} =\psi^{\mathrm{att}}_{\widetilde{a},\widetilde{b}} \left(\alpha^{-1}\left(\alpha\left(g_{\widetilde{a},\widetilde{b}}^{\circ 6ln}(\widetilde{a}i)\right)\right)\right)
=\psi^{\mathrm{att}}_{\widetilde{a},\widetilde{b}} \left(g_{\widetilde{a},\widetilde{b}}^{\circ 6ln} (\widetilde{a}i)\right)=l.$$
(In the above chain of equalities, we have used the facts that $\widetilde{a}i$ lies on the attracting equator at $\widetilde{z}$, and $\alpha$ maps $g_{\widetilde{a},\widetilde{b}}^{\circ 6ln} (\widetilde{a}i)$ to the real line). 

But, $$(\iota\circ\psi^{\mathrm{att}}_{\widetilde{a},\widetilde{b}}\circ\iota)(g_{\widetilde{a},\widetilde{b}}^{\circ 6ln}(-\widetilde{a}i))= (\iota \circ \psi^{\mathrm{att}}_{\widetilde{a},\widetilde{b}}) (g_{\widetilde{a},\widetilde{b}}^{\circ 6ln}(\widetilde{a}i))= l.$$
This shows that $\lambda=0$, and hence, $\eta = \left(\iota\circ\psi^{\mathrm{att}}_{\widetilde{a},\widetilde{b}}\circ\iota\right)^{-1} \circ \psi^{\mathrm{att}}_{\widetilde{a},\widetilde{b}}$ on $P$.

Note that the conjugacy $(\iota\circ\psi^{\mathrm{att}}_{\widetilde{a},\widetilde{b}}\circ\iota)^{-1}\circ\psi^{\mathrm{att}}_{\widetilde{a},\widetilde{b}}$ (between restrictions of $(\iota\circ g_{\widetilde{a},\widetilde{b}}^{\circ 3n})$ on some attracting petals at $\widetilde{z}$ and $\widetilde{z}^*$, respectively) maps the unique critical point $\widetilde{a}i$ of $(\iota\circ g_{\widetilde{a},\widetilde{b}}^{\circ 3n})$ in $U$ to the unique critical point $-\widetilde{a}i$ of $(\iota\circ g_{\widetilde{a},\widetilde{b}}^{\circ 3n})$ in $\iota(U)$. Hence, we can lift $(\iota\circ\psi^{\mathrm{att}}_{\widetilde{a},\widetilde{b}}\circ\iota)^{-1}\circ\psi^{\mathrm{att}}_{\widetilde{a},\widetilde{b}}$ by the iterates of $(\iota\circ g_{\widetilde{a},\widetilde{b}}^{\circ 3n})\vert_U$ and $(\iota\circ g_{\widetilde{a},\widetilde{b}}^{\circ 3n})\vert_{\iota(U)}$ to produce a conformal conjugacy between $(\iota\circ g_{\widetilde{a},\widetilde{b}}^{\circ 3n})\vert_U$ and $(\iota\circ g_{\widetilde{a},\widetilde{b}}^{\circ 3n})\vert_{\iota(U)}$. Therefore, $\eta$ extends to $U$ as a conformal conjugacy between $(\iota\circ g_{\widetilde{a},\widetilde{b}}^{\circ 3n})\vert_U$ and $(\iota\circ g_{\widetilde{a},\widetilde{b}}^{\circ 3n})\vert_{\iota(U)}$, and hence between $g_{\widetilde{a},\widetilde{b}}^{\circ 6n}\vert_U$ and $g_{\widetilde{a},\widetilde{b}}^{\circ 6n}\vert_{\iota(U)}$. 

Abusing notations, let us denote the extended conjugacy from $U \cup V$ onto $\iota\left(U \cup V\right)$ by $\eta$. Our next goal is to extend $\eta$ to a neighborhood of $\overline{U}$ (the topological closure of $U$). To this end, first observe that the basin boundaries are locally connected, and hence by Carath\'eodory's theorem, the conformal conjugacy $\eta$ extends as a homeomorphism from $\partial U$ onto $\partial \iota(U)$. Moreover, $\eta$ extends analytically across the point $\widetilde{z}$. By Montel's theorem, we have that $$\bigcup_{k} g_{\widetilde{a},\widetilde{b}}^{\circ 6kn}\left( V \cap \partial U\right) = \partial U.$$ As none of the $g_{\widetilde{a},\widetilde{b}}^{\circ 6kn}$ has a critical point on $\partial U$, we can extend $\eta$ in a neighborhood of each point of $\partial U$ by simply using the functional equation $\eta \circ g_{\widetilde{a},\widetilde{b}}^{\circ 6n} = g_{\widetilde{a},\widetilde{b}}^{\circ 6n} \circ \eta$. Since all of these extensions at various points of $\partial U$ extend the already defined (and conformal) common map $\eta$, the uniqueness of analytic continuations yields an analytic extension of $\eta$ in a neighborhood of $\overline{U}$. By construction, this extension is a proper holomorphic map, and assumes every point in $\iota(U)$ precisely once. Therefore, the extended map $\eta$ from a neighborhood of $\overline{U}$ onto a neighborhood of $\overline{\iota(U)}$ has degree $1$, and hence is our desired conformal conjugacy between polynomial-like restrictions of $g_{\widetilde{a},\widetilde{b}}^{\circ 6n}$ on some neighborhoods of $\overline{U}$ and $\overline{\iota(U)}$ (respectively).
\end{proof}

\begin{remark}
We would like to emphasize, albeit at the risk of being pedantic, that the germ conjugacy $\eta$ extends to the closure of the basins only because it respects the dynamics on the critical orbits. The map $g_{\widetilde{a},\widetilde{b}}^{\circ 6n}$ has three critical points $u$, $v$, $\widetilde{a}i$ in the Fatou component $U$, such that $g_{\widetilde{a},\widetilde{b}}^{\circ 3n}(u)=g_{\widetilde{a},\widetilde{b}}^{\circ 3n}(v)=-\widetilde{a}i$. Thus, $g_{\widetilde{a},\widetilde{b}}^{\circ 6n}$ has two infinite critical orbits in $U$; namely, $$\widetilde{a}i\xmapsto{g_{\widetilde{a},\widetilde{b}}^{\circ 6n}} g_{\widetilde{a},\widetilde{b}}^{\circ 6n}(\widetilde{a}i)\xmapsto{g_{\widetilde{a},\widetilde{b}}^{\circ 6n}}\cdots,\qquad \textrm{and}\ \qquad u,v\xmapsto{g_{\widetilde{a},\widetilde{b}}^{\circ 6n}} g_{\widetilde{a},\widetilde{b}}^{\circ 3n}(-\widetilde{a}i)\xmapsto{g_{\widetilde{a},\widetilde{b}}^{\circ 6n}}\cdots.$$ Clearly, these two critical orbits are dynamically different. By real symmetry, $g_{\widetilde{a},\widetilde{b}}^{\circ 6n}$ has three critical points $u^*, v^*, -\widetilde{a}i$ in $\iota(U)$. The corresponding critical orbits in $\iota(U)$ are given by $$-\widetilde{a}i\xmapsto{g_{\widetilde{a},\widetilde{b}}^{\circ 6n}} g_{\widetilde{a},\widetilde{b}}^{\circ 6n}(-\widetilde{a}i)\xmapsto{g_{\widetilde{a},\widetilde{b}}^{\circ 6n}}\cdots,\ \qquad  \textrm{and} \qquad u^*,v^*\xmapsto{g_{\widetilde{a},\widetilde{b}}^{\circ 6n}} g_{\widetilde{a},\widetilde{b}}^{\circ 3n}(\widetilde{a}i)\xmapsto{g_{\widetilde{a},\widetilde{b}}^{\circ 6n}}\cdots.$$ By our construction, $\eta$ maps (the tail of) each of the two (dynamically distinct) critical orbits of $g_{\widetilde{a},\widetilde{b}}^{\circ 6n}\vert_U$ to the (tail of the) corresponding critical orbit of $g_{\widetilde{a},\widetilde{b}}^{\circ 6n}\vert_{\iota(U)}$. 

It is good to keep in mind that the parabolic germs of $g_{\widetilde{a},\widetilde{b}}^{\circ 6n}$ at $\widetilde{z}$ and $\widetilde{z}^*$ are \emph{always} conformally conjugate by $g_{\widetilde{a},\widetilde{b}}^{\circ 3n}$; but this local conjugacy exchanges the two dynamically marked critical orbits, which have different topological dynamics, and hence this local conjugacy has no chance of being extended to the entire parabolic basin.
\end{remark}

\begin{theorem}[Umbilical Cord Wiggling in Real Cubics]\label{real_germ}
There does not exist a path $\gamma:\left[0,\delta\right]\to \mathbb{R}^2$ such that $\gamma(0)=(\widetilde{a},\widetilde{b})$, and $\gamma(\left(0,\delta\right]) \subset \mathcal{C}(\mathcal{G})\setminus \overline{H'}$.
\end{theorem}

\begin{figure}[ht!]
\begin{center}
\includegraphics[width=0.48\linewidth]{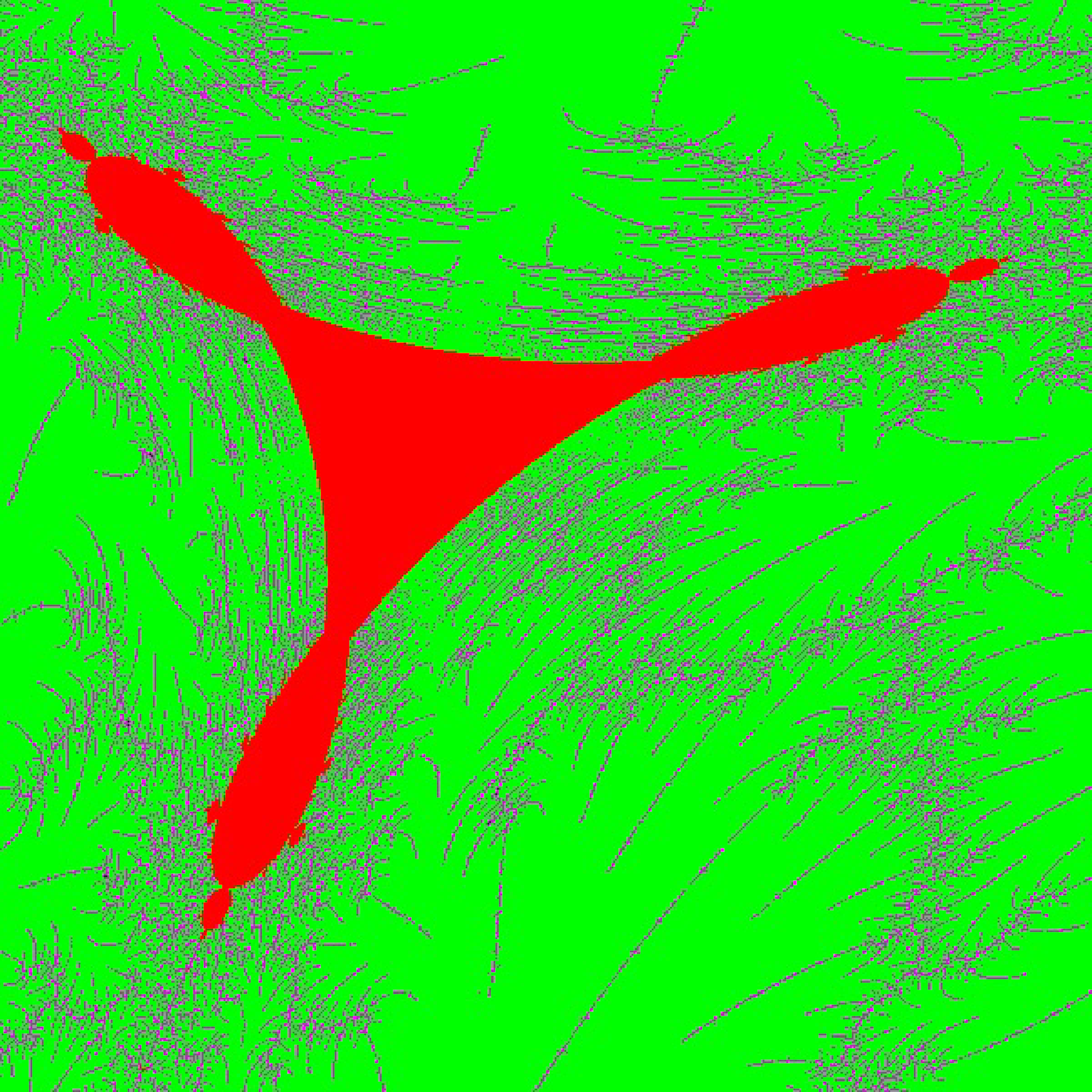}
\end{center}
\caption{Wiggling of an umbilical cord for a Tricorn-like set in the real cubic locus.}
\label{umbilical_wiggle_cubic}
\end{figure}

\begin{proof}
We have already showed in Lemma~\ref{umbilical_cubic} that the existence of such a path $\gamma$ would imply that the polynomial-like restrictions $g_{\widetilde{a},\widetilde{b}}^{\circ 6n}: U^{'} \rightarrow g_{\widetilde{a},\widetilde{b}}^{\circ 6n}(U^{'})$ (where $U^{'}$ is a neighborhood of $\overline{U}$), and $g_{\widetilde{a},\widetilde{b}}^{\circ 6n}: \iota(U^{'})\rightarrow g_{\widetilde{a},\widetilde{b}}^{\circ 6n}(\iota(U^{'}))$ are conformally conjugate. Applying \cite[Theorem~1]{I2} to this situation, we obtain polynomials $p$, $p_1$ and $p_2$ such that 
\begin{align*}
g_{\widetilde{a},\widetilde{b}}^{\circ 6n} \circ p_1 = p_1 \circ p,\ g_{\widetilde{a},\widetilde{b}}^{\circ 6n}\circ p_2 = p_2 \circ p.
\end{align*}
Moreover, since the product dynamics $\left(g_{\widetilde{a},\widetilde{b}}^{\circ 6n},\ g_{\widetilde{a},\widetilde{b}}^{\circ 6n}\right)$ is globally self-conjugate by $(\iota\times\iota)\circ q$, where $q:\mathbb{C}^2\rightarrow\mathbb{C}^2,\ q(z,\ w)=(w,\ z)$, it follows from the proof of \cite[Theorem~1]{I2} that $\Deg p_1 = \Deg p_2.$

Moreover, by Theorem \cite[Theorem~1]{I2}, $p$ has a polynomial-like restriction $p: V\rightarrow p(V)$ which is conformally conjugate to $g_{\widetilde{a},\widetilde{b}}^{\circ 6n}: U^{'}\rightarrow g_{\widetilde{a},\widetilde{b}}^{\circ 6n}(U^{'})$ by $p_1$, and to $g_{\widetilde{a},\widetilde{b}}^{\circ 6n}: \iota(U^{'})\rightarrow g_{\widetilde{a},\widetilde{b}}^{\circ 6n}(\iota(U^{'}))$ by $p_2$. We now consider two cases.

\noindent \textbf{Case 1: \boldmath{$\Deg (p_1)=\Deg (p_2)=1$}.}\quad
Set $p_3:=p_1\circ p_2^{-1}$. Then $p_3$ is an affine map commuting with $g_{\widetilde{a},\widetilde{b}}^{\circ 6n}$, and conjugating the two polynomial-like restrictions of $g_{\widetilde{a},\widetilde{b}}^{\circ 6n}$ under consideration. Clearly, $p_3\neq \mathrm{id}$. An easy computation (using the fact that $g_{\widetilde{a},\widetilde{b}}^{\circ 6n}$ is a centered real polynomial) now shows that $p_3(z)=-z$, and hence $\widetilde{b}=0$. 

\noindent \textbf{Case 2: \boldmath{$\Deg (p_1)=\Deg (p_2)=k>1$}.}\quad
We will first prove by contradiction that $\gcd(\Deg g_{\widetilde{a},\widetilde{b}}^{\circ 6n}, \Deg p_1) > 1$. To do this, let $\gcd(\Deg g_{\widetilde{a},\widetilde{b}}^{\circ 6n}, \Deg p_1) = 1$. Now we can apply \cite[Theorem~8]{I2} to our situation. Since $g_{\widetilde{a},\widetilde{b}}^{\circ 6n}$ is parabolic, it is neither a power map, nor a Chebyshev polynomial. Hence, there exists some non-constant polynomial $P$ such that $g_{\widetilde{a},\widetilde{b}}^{\circ 6n}$ is affinely conjugate to the polynomial $h(z):=z^r(P(z))^k$, and $p_1(z)=z^k$ (up to affine conjugacy). If $r\geq 2$, then $h(z)$ has a super-attracting fixed point at $0$. But $g_{\widetilde{a},\widetilde{b}}^{\circ 6n}$, which is affinely conjugate to $h(z)$, has no super-attracting fixed point. Hence, $r =0$ or $1$. By degree consideration, we have $3^{6n} = r + ks$, where $\Deg P = s$. The assumption  $\gcd(\Deg g_{\widetilde{a},\widetilde{b}}^{\circ 6n}, \Deg p_1)=\gcd(3^{6n}, k)= 1$ implies that $r=1$, i.e.,\ $h(z)=z(P(z))^k$. Now the fixed point $0$ for $h$ satisfies $h^{-1}(0)=\{ 0\} \cup P^{-1}(0)$, and any point in $P^{-1}(0)$ has a local mapping degree $k$ under $h$. The same must hold for the affinely conjugate polynomial $g_{\widetilde{a},\widetilde{b}}^{\circ 6n}$: there exists a fixed point (say $x$) for $g_{\widetilde{a},\widetilde{b}}^{\circ 6n}$ such that any point in $(g_{\widetilde{a},\widetilde{b}}^{\circ 6n})^{-1}(x)$ has mapping degree $k$ (possibly) except for $x$; in particular, all points in $(g_{\widetilde{a},\widetilde{b}}^{\circ 6n})^{-1}(x)\setminus \{ x \}$ are critical points for $g_{\widetilde{a},\widetilde{b}}^{\circ 6n}$ (since $k>1$). But this implies that $g_{\widetilde{a},\widetilde{b}}^{\circ 6n}$ has a finite critical orbit, which is a contradiction to the fact that all critical orbits of $g_{\widetilde{a},\widetilde{b}}^{\circ 6n}$ non-trivially converge to parabolic fixed points.

Now applying Engstrom's theorem \cite{Eng} (see also \cite[Theorem~11, Corollary~12, Lemma~13]{I2}), there exist polynomials $\alpha_1, \beta_1, p_{1,1}$ such that
\begin{align*}
 p &= \beta_1 \circ \alpha_1, & 
 p_{1,1} \circ (\alpha_1 \circ \beta_1) &= g_{\widetilde{a},\widetilde{b}}^{\circ 6n} \circ p_{1,1}, \\
 p_1 &= p_{1,1} \circ \alpha_1, &
 \deg \alpha_1 &= \gcd(\deg p , \deg p_1).
\end{align*}
In particular, $\alpha_1 \circ \beta_1$ is semiconjugate to $g_{\widetilde{a},\widetilde{b}}^{\circ 6n}$ by $p_{1,1}$ with $\deg p_{1,1} < \deg p_1$.
By repeating the argument, there are polynomials $\alpha_j, \beta_j, p_{1,j}$ ($j=2,\dots N$) such that
\begin{align*}
 \alpha_{j-1} \circ \beta_{j-1} &= \beta_j \circ \alpha_j, & 
 p_{1,j} \circ (\alpha_j \circ \beta_j) &= g_{\widetilde{a},\widetilde{b}}^{\circ 6n} \circ p_{1,j}, \\
 p_{1,j-1} &= p_{1,j} \circ \alpha_j, &
 \deg \alpha_j &= \gcd(\deg p_{1,j-1} , \deg p),
\end{align*}
and $\gcd(\deg p_{1,N}, \deg p) =1$.
Note that $\deg \alpha_j \circ \beta_j = \deg p$.
Then by the same argument as above, we have $\deg p_{1,N}=1$, i.e.,\ $\alpha_N \circ \beta_N$ is affinely conjugate to $g_{\widetilde{a},\widetilde{b}}^{\circ 6n}$,
so we may assume they are equal indeed.

Since $g_{\widetilde{a},\widetilde{b}}$ is a prime polynomial under composition (since its degree is a prime number),
the following \emph{chain} between $p$ and $g_{\widetilde{a},\widetilde{b}}^{\circ 6n}$
\begin{align}
 \label{eqn:chain}
 p &= \beta_1 \circ \alpha_1, &
 \alpha_{j-1} \circ \beta_{j-1} &= \beta_j \circ \alpha_j &
 \alpha_N \circ \beta_N = g_{\widetilde{a}, \widetilde{b}}^{\circ 6n}
\end{align}
implies that every $\alpha_j \circ \beta_j$ is affinely conjugate to $g_{\widetilde{a},\widetilde{b}}^{\circ 6n}$, and so is $p$.


Therefore, $p_1$ commutes with $g_{\widetilde{a},\widetilde{b}}^{\circ 6n}$. As $g_{\widetilde{a},\widetilde{b}}^{\circ 6n}$ is neither a power map nor a Chebyshev polynomial, $p_1= g_{\widetilde{a},\widetilde{b}}^{\circ k_1}$, for some $k_1\in \mathbb{N}$ (up to affine conjugacy). The same is true for $p_2$ as well; i.e.,\ $p_2= g_{\widetilde{a},\widetilde{b}}^{\circ k_1}$ (up to affine conjugacy). 

Therefore, there is a polynomial-like restriction of $p=g_{\widetilde{a},\widetilde{b}}^{\circ 6n}: V\rightarrow g_{\widetilde{a},\widetilde{b}}^{\circ 6n}(V)$, which is conformally conjugate to $g_{\widetilde{a},\widetilde{b}}^{\circ 6n}: U^{'}\rightarrow g_{\widetilde{a},\widetilde{b}}^{\circ 6n}(U^{'})$ by $p_1=g_{\widetilde{a},\widetilde{b}}^{\circ k_1}$, and to $g_{\widetilde{a},\widetilde{b}}^{\circ 6n}: \iota(U^{'})\rightarrow g_{\widetilde{a},\widetilde{b}}^{\circ 6n}(\iota(U^{'}))$ by $p_2=g_{\widetilde{a},\widetilde{b}}^{\circ k_1}$. But the dynamical configuration implies that this is impossible (since there is only one parabolic cycle, and the unique cycle of immediate parabolic basins contains two critical points of $g_{\widetilde{a},\widetilde{b}}$, either $p_1$ or $p_2$ must have a critical point in their corresponding conjugating domain).

Therefore, we have showed that the existence of such a path $p$ would imply that $\widetilde{b}=0$. But this contradicts our assumption that $H^{\prime}$ does not intersect the line $\{b=0\}$. This completes the proof of the theorem.
\end{proof}

Using Theorem~\ref{real_germ}, we can now proceed to prove that the straightening map $\chi_{a_0,b_0}: \mathcal{R}(a_0,b_0) \to \mathcal{M}_2^*$ is discontinuous.

\begin{proof}[Proof of Theorem~\ref{Straightening_discontinuity_2}]
We will stick to the terminologies used throughout this section. We will assume that the map $\chi_{a_0,b_0}: \mathcal{R}(a_0,b_0) \to \mathcal{M}_2^*$ is continuous, and arrive at a contradiction. Due to technical reasons, we will split the proof in two cases.

\noindent \textbf{Case 1: \boldmath{$n>1$}.} \quad We have observed that when $n$ is larger than one, $\mathcal{R}(a_0, b_0)$ is compact. Moreover, the map $\chi_{a_0,b_0}$ is injective. Since an injective continuous map from a compact topological space onto a Hausdorff topological space is a homeomorphism, it follows that $\chi_{a_0,b_0}$ is a homeomorphism from $\mathcal{R}(a_0, b_0)$ onto its range (we do not claim that $\chi_{a_0, b_0}(\mathcal{R}(a_0, b_0)) = \mathcal{M}_2^*$). In particular, $\chi_{a_0, b_0}(\mathcal{R}(a_0, b_0))$ is closed. Since real hyperbolic quadratic polynomials are dense in $\mathbb{R}\cap\mathcal{M}_2^*$ \cite{GS, L4} (the Tricorn and the Mandelbrot set agree on the real line), it follows from Proposition~\ref{prop:real cubic onto hyp} and the $3$-fold rotational symmetry of the Tricorn that $\left(\mathbb{R}\cup \omega\mathbb{R}\cup \omega^2\mathbb{R}\right)\cap \mathcal{M}_2^* \subset \chi_{a_0, b_0}(\mathcal{R}(a_0, b_0))$ (where $\omega=e^{\frac{2\pi i}{3}}$).

Note that $-1.75$ is a parabolic parameter that lies on the boundary of the real period $3$ `airplane' component of the Tricorn, and $-1.25$ is the root of the real period $4$ component of the Tricorn that bifurcates from the real period $2$ `basilica' component. Clearly, $(-1.75,-1.25]\subset\mathcal{M}_2^*$, and $(-1.75,-1.25]$ is disjoint from the unique real period $3$ hyperbolic component of the Tricorn. Setting $\gamma$ as $(-1.75,-1.25]$ or one of its rotates by angle $\pm2\pi/3$, we conclude that $\gamma$ is an arc in $\mathcal{M}_2^*$ that lies outside of $\overline{H_i}$ (where $H_i=\chi_{a_0,b_0}(H')$), and lands at the critical {\'E}calle height $0$ parameter on the root parabolic arc of $\partial H_i$. By our assumption, $\chi_{a_0, b_0}$ is a homeomorphism such that $\gamma\subset\chi_{a_0,b_0}(\mathcal{R}(a_0,b_0))$; and hence the curve $\chi_{a_0, b_0}^{-1}\left(\gamma\right)$ lies in the exterior of $\overline{H^{\prime}}$, and lands at the critical {\'E}calle height $0$ parameter on the root arc of $\partial H^{\prime}$ (critical {\'E}calle heights are preserved by hybrid equivalences). But this contradicts Theorem~\ref{real_germ}, and proves the theorem for $n>1$.

\noindent \textbf{Case 2: \boldmath{$n=1$}.} \quad Finally we look at $(a_0,b_0)=(\frac{1}{\sqrt{2}},0)$. Note that since $g_{a_0,b_0}$ is not primitive in this case, $\mathcal{R}(a_0, b_0)$ is not compact. So we cannot use the arguments of Case 1 directly, and we have to work harder to demonstrate that the image of the straightening map contains a suitable interval of the real line.

In the dynamical plane of $g_{a_0,b_0}$, the real line consists of two external rays (at angles $1/4$ and $3/4$) as well as their common landing point $0$, which is the unique real fixed point of $g_{a_0,b_0}$. Recall that the rational lamination of every polynomial in $\mathcal{R}(a_0, b_0)$ is stronger than that of $g_{a_0,b_0}$, and the dynamical $1/4$ and $3/4$-rays are always contained in the real line. Therefore in the dynamical plane of every $(a,b)\in \mathcal{R}(a_0, b_0)$, the real line consists of two external rays (at angles $1/4$ and $3/4$) as well as their common landing point which is repelling. In order to obtain a period $1$ renormalization for any polynomial in $\mathcal{R}(a_0, b_0)$, one simply has to perform a standard Yoccoz puzzle construction starting with the $1/4$ and $3/4$ rays, and then thicken the depth $1$ puzzle (for construction of Yoccoz puzzles and the thickening procedure which yields compact containment of the domain of the polynomial-like map in its range, see \cite[p.~82]{M5}). Now, the only possibility of having a non-renormalizable map as a limit of maps in $\mathcal{R}(a_0, b_0)$ is if the dynamical  $1/4$ and $3/4$-rays land at parabolic points. This can happen in two different ways. If these two rays land at a common parabolic point (since such a parabolic fixed point would have two petals, it would prohibit the thickening procedure), then by Lemma~\ref{indiff_fixed}, the multiplier of the parabolic fixed point must be $-1$. On the other hand, if the dynamical $1/4$ and $3/4$-rays land at two distinct parabolic points, then those parabolic points would form a $2$-cycle with multiplier $+1$ (the conclusion about the multiplier follows from the fact that the first return map fixes each dynamical ray). Therefore, $\overline{\mathcal{R}(a_0, b_0)}\setminus\mathcal{R}(a_0, b_0) \subset \per_1(-1)\cup\per_2(1)$.

As in the previous case, there exists a curve $\gamma\subset \mathcal{M}_2^*$ that lies outside of $\overline{H_i}$ and lands at the critical {\'E}calle height $0$ parameter on the root parabolic arc of $\partial H_i$. Moreover, $H_i$ is in the range of $\chi_{a_0, b_0}$, and $H'=\chi_{a_0, b_0}^{-1}(H_i)$ does not intersect the line $\{b=0\}$. To complete the proof of the theorem, it suffices to show that there is a compact set $K \subset\mathcal{R}(a_0, b_0)$ with $\overline{H_i}\cup \gamma\subset\chi_{a_0,b_0}(K)$. Indeed, if there exists such a set $K$, then $\chi_{a_0, b_0}\vert_K$ would be a homeomorphism (recall that $\chi_{a_0, b_0}$ is continuous by assumption). Therefore, the curve $\chi_{a_0, b_0}^{-1}(\gamma)$ would lie in the exterior of $\overline{H^{\prime}}$, and land at the critical {\'E}calle height $0$ parameter on the root arc of $\partial H^{\prime}$. Once again, this contradicts Theorem~\ref{real_germ}, and completes the proof in the $n=1$ case.

Let us now prove the existence of the required compact set $K$. Note that since $\overline{H'}$ is contained in the union of the hyperbolicity locus and $\per_6(1)$ of the family $\mathcal{G}$, it follows that $\overline{H'}$ is disjoint from $\per_1(-1)\cup\per_2(1)$. Hence $\overline{H'}$ is contained in a compact subset of $\mathcal{R}(a_0, b_0)$.

Let us denote the hyperbolic parameters of $\gamma$ by $\gamma^{\mathrm{hyp}}$. By Lemma~\ref{prop:real cubic onto hyp}, $\gamma^{\mathrm{hyp}}$ is contained in the range of $\chi_{a_0,b_0}$. We will now show that $\chi_{a_0,b_0}^{-1}(\gamma^{\mathrm{hyp}})$ does not accumulate on $\per_1(-1)\cup\per_2(1)$; i.e.,\ $\chi_{a_0,b_0}^{-1}(\gamma^{\mathrm{hyp}})$ is contained in a compact subset of $\mathcal{R}(a_0, b_0)$. To this end, observe that $\gamma^{\mathrm{hyp}}$ is contained in the $1/2$-limb of a period $2$ hyperbolic component of $\mathcal{M}_2^*$. So for each parameter on $\gamma^{\mathrm{hyp}}$, two $4$-periodic dynamical rays land at a common point of the corresponding Julia set. Hence for each parameter on $\chi_{a_0,b_0}^{-1}(\gamma^{\mathrm{hyp}})$, two $4$-periodic dynamical rays (e.g. at angles $61/80$ and $69/80$) land at a common point. If $\chi_{a_0,b_0}^{-1}(\gamma^{\mathrm{hyp}})$ accumulates on some parameter on the parabolic curves $\per_1(-1)\cup \per_2(1)$, then the corresponding dynamical rays at angles $61/80$ and $69/80$ would have to co-land in the dynamical plane of that parameter. But there is no such landing relation for parameters on $\per_1(-1)\cup\per_2(1)$. This proves that $\chi_{a_0,b_0}^{-1}(\gamma^{\mathrm{hyp}})$ is contained in a compact subset of $\mathcal{R}(a_0, b_0)$. 

Combining the observations of the previous two paragraphs, we conclude that there is a compact subset $K$ of $\mathcal{R}(a_0, b_0)$ that contains $\overline{H'}\cup\chi_{a_0,b_0}^{-1}(\gamma^{\mathrm{hyp}})$. Since we assumed $\chi_{a_0, b_0}$ to be continuous, it follows that $\chi_{a_0, b_0}(K)$ is a closed set containing $\gamma^{\mathrm{hyp}}$. But $\gamma^{\mathrm{hyp}}$ is dense in $\gamma$ (by the density of hyperbolic quadratic polynomials in $\mathbb{R}$). Therefore, $\gamma\subset\chi_{a_0, b_0}(K)$. Therefore, $K$ is the required compact subset of $\mathcal{R}(a_0, b_0)$ such that $\chi_{a_0,b_0}(K)\supset\overline{H_i}\cup\gamma$.
\end{proof}

\section{Recovering Unicritical Maps from Their Parabolic Germs}\label{cauliflower_recover}

Recall that one of the key steps in the proof of Theorem~\ref{Straightening_discontinuity_2} (more precisely, in the proof of Lemma~\ref{umbilical_cubic}) was to extend a carefully constructed local (germ) conjugacy to a semi-local (polynomial-like map) conjugacy, which allowed us to conclude that the corresponding polynomials are affinely conjugate. The extension of the germ conjugacy made use of some of its special properties; in particular, we used the fact that the germ conjugacy preserves the post-critical orbits. However, in general, a conjugacy between two polynomial parabolic germs has no reason to preserve the post-critical orbits (germ conjugacies are defined locally, and post-critical orbits are global objects).

Motivated by the above discussion, we will prove a rigidity property for unicritical holomorphic and anti-holomorphic parabolic polynomials in this section (which answers \cite[Question~3.6]{IM4} for unicritical polynomials). In particular, we will show that a unicritical holomorphic polynomial having a parabolic cycle is completely determined by the conformal conjugacy class of its parabolic germ or equivalently, by its {\'E}calle-Voronin invariants. 

We will need the concept of extended horn maps, which are the natural maximal extensions of horn maps. For the sake of completeness, we include the basic definitions and properties of horn maps. For simplicity, we will only define it in the context of parabolic points with multiplier $1$, and a single petal. More comprehensive accounts on these ideas can be found in \cite[\S 2]{BE}.

Let $p$ be a (parabolic) holomorphic polynomial, $z_0$ be such that $p^{\circ k}(z_0)=z_0$, and $p^{\circ k}(z) = z+(z-z_0)^2+\mathcal{O}(\vert z-z_0\vert^3)$ locally near $z_0$. The parabolic point $z_0$ of $p$ has exactly two petals, one attracting and one repelling (denoted by $\mathcal{P}^{\textrm{att}}$ and $\mathcal{P}^{\textrm{rep}}$ respectively). The intersection of the two petals has two connected components. We denote by $\mathcal{U}^+$ the connected component of $\mathcal{P}^{\textrm{att}} \cap \mathcal{P}^{\textrm{rep}}$ whose image under the Fatou coordinates is contained in the upper half-plane, and by $\mathcal{U}^-$ the one whose image under the Fatou coordinates is contained in the lower half-plane. We define the ``sepals'' $\mathcal{S}^{\pm}$ by
\begin{align*}
\displaystyle \mathcal{S}^{\pm} = \bigcup_{n \in \mathbb{Z}} p^{\circ nk} (\mathcal{U}^{\pm})
\end{align*}

Note that each sepal contains a connected component of the intersection of the attracting and the repelling petals, and they are invariant under the first holomorphic return map of the parabolic point. The attracting Fatou coordinate $\psi^{\textrm{att}}$ (respectively the repelling Fatou coordinate $\psi^{\textrm{rep}}$) can be extended to $\mathcal{P}^{\textrm{att}} \cup \mathcal{S}^+ \cup \mathcal{S}^-$ (respectively to $\mathcal{P}^{\textrm{rep}} \cup \mathcal{S}^+ \cup \mathcal{S}^-$) such that they conjugate the first holomorphic return map to the translation $\zeta \mapsto \zeta+1$.

\begin{definition}[Lifted horn maps]\label{lifted_horn}
Let us define $V^- = \psi^{\textrm{rep}}(\mathcal{S}^-)$, $V^+ = \psi^{\textrm{rep}}(\mathcal{S}^+)$, $W^- = \psi^{\textrm{att}}(\mathcal{S}^-)$, and $W^+ =\psi^{\textrm{att}}(\mathcal{S}^+)$. Then, denote by $H^- : V^- \rightarrow W^-$ the restriction of $ \psi^{\textrm{att}} \circ (\psi^{\textrm{rep}})^{-1}$ to $V^-$, and by $H^+: V^+ \rightarrow W^+$ the restriction of $\psi^{\textrm{att}} \circ \left(\psi^{\textrm{rep}}\right)^{-1}$ to $V^+$. We refer to $H^{\pm}$ as lifted horn maps for $p$ at $z_0$.
\end{definition}

Lifted horn maps are unique up to pre and post-composition by translation. Note that such translations must be composed with both of the $H^\pm$ at the same time.
The regions $V^{\pm}$ and $W^{\pm}$ are invariant under translation by $1$. Moreover, the asymptotic development of the Fatou coordinates implies that the regions $V^+$ and $W^+$ contain an upper half-plane, whereas the regions $V^-$ and $W^-$ contain a lower half-plane. Consequently, under the projection $\Pi : \zeta \mapsto w = \exp(2i\pi \zeta)$, the regions $V^+$ and $W^+$ project to punctured neighborhoods $\mathcal{V}^+$ and $\mathcal{W}^+$ of $0$, whereas $V^-$ and $W^-$ project to punctured neighborhoods $\mathcal{V}^-$ and $\mathcal{W}^-$ of $\infty$. 

The lifted horn maps $H^{\pm}$ satisfy $H^{\pm}(\zeta + 1) = H^{\pm}(\zeta) + 1$ on $V^{\pm}$.  Thus, they project to mappings $h^{\pm} : \mathcal{V}^{\pm} \rightarrow \mathcal{W}^{\pm}$ such that the following diagram commutes:

\begin{center}
$\begin{CD}
V^{\pm} @>H^{\pm}>> W^{\pm}\\
@VV\Pi V @VV\Pi V\\
\mathcal{V}^{\pm} @>h^{\pm}>> \mathcal{W}^{\pm}
\end{CD}$
\end{center}

It is well-known that $\exists$ $\eta, \eta' \in \mathbb{C}$ such that $H^+(\zeta) \approx \zeta + \eta$ when $\im(\zeta) \rightarrow +\infty$, and $H^-(\zeta) \approx \zeta + \eta'$ when $\im(\zeta) \rightarrow -\infty$. This proves that $h^+ (w) \rightarrow 0$ as $w \rightarrow 0$. Thus, $h^+$ extends analytically to $0$ by $h^+ (0) = 0$. One can show similarly that $h^-$ extends analytically to $\infty$ by $h^- (\infty) = \infty$.

\begin{definition}[Horn Maps]\label{horn}
The maps $h^{+}: \mathcal{V}^+\cup \{ 0\}\rightarrow \mathcal{W}^+\cup \{ 0\}$, and $h^{-}: \mathcal{V}^-\cup \{ \infty\}\rightarrow \mathcal{W}^-\cup \{ \infty\}$ are called \emph{horn maps} for $p$ at $z_0$.
\end{definition}

Let $U_0$ be the immediate basin of attraction of $z_0$. Then there exists an extended attracting Fatou coordinate $\psi^{\mathrm{att}}:U_0\to \mathbb{C}$ (which is a ramified covering ramified only over the pre-critical points of $p^{\circ k}$ in $U_0$) satisfying $\psi^{\mathrm{att}}(p^{\circ k}(z))=\psi^{\mathrm{att}}(z)+1$, for every $z\in U_0$ (compare Figure~\ref{chessboard}). Similarly, the inverse of the repelling Fatou coordinate $\psi^{\mathrm{rep}}$ at $z_0$ extends to a holomorphic map $\zeta^{\mathrm{rep}}:\mathbb{C}\to\mathbb{C}$ satisfying $p^{\circ k}(\zeta^{\mathrm{rep}}(\zeta))=\zeta^{\mathrm{rep}}(\zeta+1)$, for every $\zeta\in \mathbb{C}$. We define $D_0^+$ (respectively $D_0^-$) to be the connected component of $\left(\zeta^{\mathrm{rep}}\right)^{-1}(U_0)$ containing an upper half plane (respectively a lower half plane). Furthermore, let $\mathcal{D}_0^{\pm}$ be the image of $D_0^{\pm}$ under the projection $\Pi: \zeta\mapsto w=\exp(2i\pi\zeta)$.

\begin{definition}[Extended Horn Map]
The maps $H^{\pm}:= \psi^{\mathrm{att}}\circ \zeta^{\mathrm{rep}}: D_0^{\pm}\to \mathbb{C}$ are called the \emph{extended lifted horn maps} for $p$ at $z_0$. They project (under $\Pi$) to the holomorphic maps $h^{\pm}:\mathcal{D}_0^{\pm}\to \hat{\mathbb{C}}$, which are called the \emph{extended horn maps} for $p$ at $z_0$.
\end{definition}

\begin{figure}
\begin{center}
\includegraphics[width=0.48\linewidth]{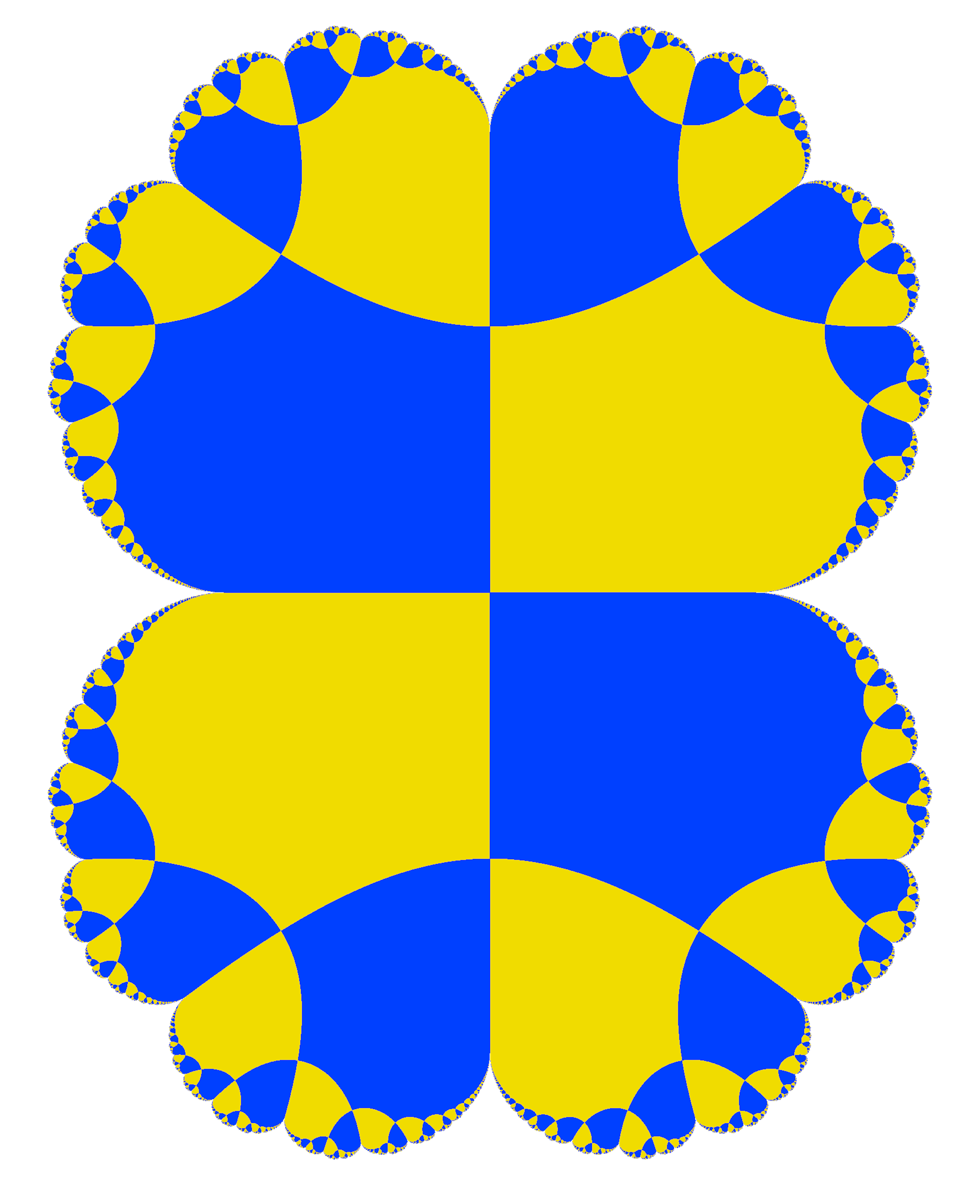}
\end{center}
\caption{The parabolic chessboard for the polynomial $z+z^2$: normalizing $\psi^{\mathrm{att}}(-\frac{1}{2})=0$, each yellow tile biholomorphically maps to the upper half plane, and each blue tile biholomorphically maps to the lower half plane under $\psi^{\mathrm{att}}$. The pre-critical points of $z+z^2$ or equivalently the critical points of  $\psi^{\mathrm{att}}$ are located where four tiles meet (Figure Courtesy Arnaud Ch{\'e}ritat).}
\label{chessboard}
\end{figure}

We will mostly work with the horn map $h^+ : \mathcal{D}_0^+ \to \hat{\mathbb{C}}$. Note that $\mathcal{D}_0^+\cup\{0\}$ is the maximal domain of analyticity of the map $h^+$. This can be seen as follows (see \cite[Theorem~2.31]{LY} for a more general assertion of this type). Let $z' \in \partial U_0$, then there exists a sequence of pre-parabolic points $\{ z_n\}_{n\geq 1}\subset \partial U_0$ converging to $z'$ such that for each $n$, there is an arc $\gamma_n:\left(0,1\right]\to U_0$ with $\gamma_n(0)=z_n$ satisfying the properties $\re(\psi^{\mathrm{att}}(\gamma_n\left(0,1\right]))=$ constant, and $\displaystyle \lim_{s\downarrow 0} \im(\psi^{\mathrm{att}}(\gamma_n(s)))=+\infty$. Therefore, for every $w'\in \partial \mathcal{D}_0^+$, there exists a sequence of points $\{ w_n\}_{n\geq 1}\subset \partial \mathcal{D}_0^+$ converging to $w'$ such that for each $n$, there is an arc $\Gamma_n:(0,1]\to \mathcal{D}_0^+$ with $\Gamma_n(0)=w_n$ satisfying $\displaystyle \lim_{s\downarrow 0} h^+(\Gamma_n(s))=0$. It follows from the identity principle that if we could continue $h^+$ analytically in a neighborhood of $w'$, then $h^+$ would be identically $0$, which is a contradiction to the fact that $h^+$ is asymptotically a rotation near $0$.  

Let us now recall the definitions of some basic objects for the multibrot set.

\begin{definition}[Multibrot Sets]
The \emph{multibrot set} of degree $d$ is defined as $$\mathcal{M}_d = \lbrace c \in \mathbb{C} : K(p_c)\ \mathrm{is\ connected}\rbrace,$$ where $K(p_c)$ is the filled Julia set of the unicritical holomorphic polynomial $p_c(z)=z^d+c$.
\end{definition}

Recall that for a polynomial $p_c$ with a parabolic cycle, the characteristic Fatou component of $p_c$ is defined as the unique Fatou component of $p_c$ containing the critical value $c$. The characteristic parabolic point of $p_c$ is defined as the unique parabolic point on the boundary of the characteristic Fatou component. 

A parabolic parameter $c$ lying on the boundary of a period $n$ hyperbolic component $H$ of $\mathcal{M}_d$ is called the \emph{root} of $H$ if the characteristic Fatou component of $p_c$  has period $n$, and the characteristic parabolic point of $p_c$ is a cut-point of the Julia set. On the other hand, a parabolic parameter $c$ lying on the boundary of a period $n$ hyperbolic component $H$ of $\mathcal{M}_d$ is called a \emph{co-root} of $H$ if the characteristic Fatou component of $p_c$ has period $n$, but the characteristic parabolic point of $p_c$ is not a cut-point of the Julia set. Every hyperbolic component (of period $n>1$) of $\mathcal{M}_d$ has exactly one root on its boundary. A hyperbolic component $H$ of $\mathcal{M}_d$ is said to be \emph{satellite} (respectively, \emph{primitive}) if the unique root point on its boundary lies on the boundary of another hyperbolic component (respectively, does not lie on the boundary of any other hyperbolic component).  We refer the readers to \cite{EMS} for a detailed discussion of these notions.

With these preparations, we are now ready to prove our first local-global principle for parabolic germs.

\begin{definition}
\begin{itemize}
\item Let $\mathcal{M}_d^{\mathrm{cusp}}$ be the union of the set of all root points of the primitive hyperbolic components, and the set of all co-root points of the multibrot set $\mathcal{M}_d$. For $c_1, c_2 \in \mathcal{M}_d^{\mathrm{cusp}}$, we write $c_1 \sim c_2$ if $z^d+c_1$ and $z^d+c_2$ are affinely conjugate; i.e.,\ if $c_2/c_1$ is a $(d-1)$-st root of unity. We denote the set of equivalence classes under this equivalence relation by $\left(\mathcal{M}_d^{\mathrm{cusp}}/\mathord\sim\right)$.

\item Let $\Diff^{+1}(\mathbb{C},0)$ be the set of conformal conjugacy classes of holomorphic germs (at $0$) fixing $0$, and having multiplier $+1$ at $0$.
\end{itemize}
\end{definition}

For $c\in \mathcal{M}_d^{\mathrm{cusp}}$, let $z_c$ be the characteristic parabolic point of $p_{c}(z)=z^d+c$, and $k$ be the period of $z_c$. Conjugating $p_c^{\circ k}\vert_{N_{z_c}}$ (where $N_{z_c}$ is a sufficiently small neighborhood of $z_c$) by an affine map that sends $z_c$ to the origin, one obtains an element of $\Diff^{+1}(\mathbb{C},0)$. The following lemma settles the germ rigidity for parameters in $\mathcal{M}_d^{\mathrm{cusp}}$ (i.e.,\ for parabolic parameters with a single petal).

\begin{lemma}[Parabolic Germs Determine Co-roots, and Roots of Primitive Components]\label{Parabolic_Germs_Determine_Roots}
The map 
\begin{center}
$\displaystyle \bigsqcup_{d\geq 2}\left(\mathcal{M}_d^{\mathrm{cusp}}/\mathord\sim\right) \to \Diff^{+1}(\mathbb{C},0)$\\
$c\mapsto p_c^{\circ k}\vert_{N_{z_c}}$
\end{center}
is injective.
\end{lemma}

The rough idea of the proof of Lemma~\ref{Parabolic_Germs_Determine_Roots} is the following. The assumption that $c$ is a root point of a \emph{primitive} hyperbolic component or a co-root point of a hyperbolic component of period $k$ (i.e.,\ $c\in\mathcal{M}_d^{\textrm{cusp}}$) implies that $p_c$ has exactly one attracting petal at the characteristic parabolic point $z_c$, and hence $p_c^{\circ k}$ restricts to a polynomial-like map in a neighborhood of the closure of its characteristic Fatou component. If the parabolic germs determined by $p_{c_1}^{\circ k_1}$ and $p_{c_2}^{\circ k_2}$ near their characteristic parabolic points (for some $c_i\in\mathcal{M}_{d_i}^{\mathrm{cusp}}$, $i=1,2$) are conformally conjugate, then we will first promote the conformal conjugacy between the parabolic germs to a conformal conjugacy between the polynomial-like restrictions of $p_{c_1}^{\circ k_1}$ and $p_{c_2}^{\circ k_2}$ in small neighborhoods of the closures of their characteristic Fatou components. This would allow us to apply \cite[Theorem~1]{I2} yielding certain polynomial semi-conjugacy relations between $p_{c_1}^{\circ k_1}$ and $p_{c_2}^{\circ k_2}$. Finally, a careful analysis of the semi-conjugacy relations using the reduction step of Ritt and Engstrom will give us an affine conjugacy between  $p_{c_1}$ and $p_{c_2}$.

\begin{proof}[Proof of Lemma~\ref{Parabolic_Germs_Determine_Roots}]
For $i=1,2$, let $c_i\in \mathcal{M}_{d_i}^{\mathrm{cusp}}$, the parabolic cycle of $c_i$ have period $k_i$, the characteristic parabolic points of $p_{c_i}(z)=z^{d_i}+c_i$ be $z_i$, and the characteristic Fatou components of $p_{c_i}$ be $U_{c_i}$. 

We assume that $g_1:=p_{c_1}^{\circ k_1}\vert_{N_{z_1}}$ and $g_2:=p_{c_2}^{\circ k_2}\vert_{N_{z_2}}$ are conformally conjugate by some local biholomorphism $\phi:N_{z_1}\to N_{z_2}$. Then these two germs have the same horn map germ at $0$, and hence $p_{c_1}$ and $p_{c_2}$ have the same extended horn map $h^+$ (recall that the domain of $h^+$ is its maximal domain of analyticity; i.e.,\ $h^+$ is completely determined by the germ of the horn map at $0$). If $\psi^{\mathrm{att}}_{c_2}$ is an extended attracting Fatou coordinate for $p_{c_2}$ at $z_2$, then there exists an extended attracting Fatou coordinate $\psi^{\mathrm{att}}_{c_1}$ for $p_{c_1}$ at $z_1$ such that $\psi^{\mathrm{att}}_{c_1}=\psi^{\mathrm{att}}_{c_2}\circ \phi$ in their common domain of definition. By \cite[Proposition~4]{BE}, $h^+$ is a ramified covering with the unique critical value $\Pi\left(\psi^{\mathrm{att}}_{c_1}(c_1)\right)=\Pi\left(\psi^{\mathrm{att}}_{c_2}(c_2)\right)$. Note that the ramification index of $h^+$ over this unique critical value is $d_1-1=d_2-1$. This shows that $d_1=d_2$. We set $d:=d_1=d_2$.

Furthermore, $\psi^{\mathrm{att}}_{c_1}(c_1)-\psi^{\mathrm{att}}_{c_2}(c_2)=n\in \mathbb{Z}$. We can normalize our attracting Fatou coordinates such that $\psi^{\mathrm{att}}_{c_1}(c_1)=0$, and $\psi^{\mathrm{att}}_{c_2}(c_2)=-n$. Put $\eta:= g_2^{\circ (-n)}\circ \phi$. Then, $\eta$ is a new conformal conjugacy between $g_1$ and $g_2$. We stick to the Fatou coordinate $\psi^{\mathrm{att}}_{c_1}$ for $p_{c_1}$, and define a new Fatou coordinate $\widetilde{\psi}^{\mathrm{att}}_{c_2}$ for $p_{c_2}$ such that $\psi^{\mathrm{att}}_{c_1}=\widetilde{\psi}^{\mathrm{att}}_{c_2}\circ \eta$ in their common domain of definition. Let $N$ be large enough so that $p_{c_2}^{\circ k_2(N+n)}(c_2)$ is contained in the domain of definition of $\phi^{-1}$. Now,
\begin{align*}
\widetilde{\psi}^{\mathrm{att}}_{c_2}(c_2)
&= \widetilde{\psi}^{\mathrm{att}}_{c_2}(p_{c_2}^{\circ Nk_2}(c_2))-N
\\
&= \psi^{\mathrm{att}}_{c_1}\left(\phi^{-1}\left(p_{c_2}^{\circ (N+n)k_2}(c_2)\right)\right)-N
\\
&= \psi^{\mathrm{att}}_{c_2}\left(p_{c_2}^{\circ (N+n)k_2}(c_2)\right)-N
\\
&= \psi^{\mathrm{att}}_{c_2}(c_2)+n+N-N
\\
&= 0.
\end{align*}

Therefore, we have a germ conjugacy $\eta$ such that the Fatou coordinates of $p_{c_1}$ and $p_{c_2}$ satisfy the following properties
$$
 \psi^{\mathrm{att}}_{c_1} =\widetilde{\psi}^{\mathrm{att}}_{c_2}\circ \eta, \qquad \textrm{and}\ \qquad \psi^{\mathrm{att}}_{c_1}(c_1) =0= \widetilde{\psi}^{\mathrm{att}}_{c_2}(c_2).
$$

Since the parabolic maps $p_{c_1}^{\circ k_1}\vert_{U_{c_1}}$ and $p_{c_2}^{\circ k_2}\vert_{U_{c_2}}$ have a unique critical point of the same degree, they are conformally conjugate (see \cite[Expos{\'e}~IX]{orsay} for the proof of this statement in the case when the common degree is $2$; see \cite[\S 1.5]{Ch} for the general case). One can now carry out the arguments of Lemma~\ref{umbilical_cubic} to conclude that $\eta$ extends to a conformal conjugacy between restrictions of $p_{c_1}^{\circ k_1}$ and $p_{c_2}^{\circ k_2}$ on some neighborhoods of $\overline{U_{c_1}}$ and $\overline{U_{c_2}}$ (respectively). The condition that $c_i$ is a root point of a primitive hyperbolic component or a co-root point implies that $z_i$ has exactly one attracting petal, and hence $\eta$ induces a conformal conjugacy between the polynomial-like restrictions of $p_{c_1}^{\circ k_1}$ and $p_{c_2}^{\circ k_2}$ on some neighborhoods of $\overline{U_{c_1}}$ and $\overline{U_{c_2}}$ (respectively).

We can now invoke \cite[Theorem~1]{I2} to deduce the existence of polynomials $h$, $h_1$ and $h_2$ such that 
\begin{equation}
p_{c_1}^{\circ k_1} \circ h_1 = h_1 \circ h,\ \textrm{and}\ p_{c_2}^{\circ k_2}\circ h_2 = h_2 \circ h.
\label{semi_conj_1}
\end{equation} 
In particular, we have that $\Deg (p_{c_1}^{\circ k_1})=d_1^{k_1}=d_1^{k_2}=\Deg (p_{c_2}^{\circ k_2})$. Hence, $k_1=k_2$.

If both $h_1$ and $h_2$ are of degree one, then we are done. Now suppose that $\Deg (h_i)>1$ for some $i$. Since $p_{c_i}^{\circ k_i}$ has no finite critical orbit, the arguments used in Case 2 of the proof of Theorem~\ref{real_germ} apply mutatis mutandis to show that $\gcd(\Deg (p_{c_i}^{\circ k_i}),\Deg (h_i))>1$
and there exist chains (as in Equation \eqref{eqn:chain}) between $h$ and $p_{c_i}^{\circ k_i}$ ($i=1,2$). By Ritt's decomposition theorem \cite{R},
there exist polynomials $\alpha_i$, $\beta_i$ (of degree at least two) such that up to affine conjugacy,
\begin{equation}
 \label{semiconjugacy}
 p_{c_i}^{\circ k_i}= \alpha_i \circ \beta_i
 \text{ and }
 h = \beta_i \circ \alpha_i.
\end{equation}
Note that each prime factor in the decomposition of $z^d+c$ is either a power map with prime power, or a unicritical polynomial $z^m+c$ where $m$ is a prime divisor of $d$.

Without loss of generality, we are now led to two different cases.

\noindent \textbf{Case 1: \boldmath{$\Deg (h_i)>1\ \textrm{for}\ i=1,2$}.}\quad
In this case, Equation~\eqref{semiconjugacy} provides us with polynomials $\alpha_i$, $\beta_i$, $i=1,2$ (of degree at least two) such that up to affine conjugacy,
\begin{equation}
 \label{semiconjugacy_1}
 p_{c_1}^{\circ k_1}= \alpha_1 \circ \beta_1,\quad
 p_{c_2}^{\circ k_2}= \alpha_2 \circ \beta_2,\quad
 h = \beta_1 \circ \alpha_1=\beta_2\circ \alpha_2,\quad
\end{equation}

By decomposing the above equations to prime factors, 
we have, by taking affine conjugacy, that
\[
 h(z) = z^{l_i} \circ p_{c_i}^{k_1-1} \circ (z^{m_i}+c_i)
\]
with $l_im_i = d_i$. Therefore, it follows that $l_1=l_2$, $m_1=m_2$
and $c_1=c_2$.


\noindent \textbf{Case 2: \boldmath{$\Deg (h_1)=1\ \textrm{and}\ \Deg (h_2)>1$}.}\quad
In this case, Equation~\eqref{semiconjugacy} provides us with polynomials $\alpha$, $\beta$ (of degree at least two) such that up to affine conjugacy,
\begin{equation}
 \label{semiconjugacy_2}
 p_{c_1}^{\circ k_1}= \alpha\circ \beta,\quad
 p_{c_2}^{\circ k_2}= \beta \circ \alpha,\quad
 \end{equation}

Once again, using the fact that each prime factor in the decomposition of $z^d+c$ is either a power map with prime power, or a unicritical polynomial $z^m+c$ where $m$ is a prime divisor of $d$, we conclude that $\alpha$ and $\beta$ must be iterates of $p_{c_1}$, and hence $p_{c_1}^{\circ k_1}=
 p_{c_2}^{\circ k_2}$ up to affine conjugacy. Therefore, $p_{c_1}$ and $p_{c_2}$ are also affinely conjugate.
\end{proof}

Note that the proof of Lemma~\ref{Parabolic_Germs_Determine_Roots} roughly consists of an analytic part and an algebraic part. The analytic part was to promote the conjugacy between parabolic germs to a conformal conjugacy between suitable polynomial-like maps. Thanks to \cite[Theorem~1]{I2}, this gave rise to the semi-conjugacy relations~\eqref{semi_conj_1}. The next step, where we used the work of Ritt and Engstrom to obtain an affine conjugacy between the polynomials $p_{c_1}$ and $p_{c_2}$ was purely algebraic. In fact, the only conditions on $p_{c_1}$ and $p_{c_2}$ that we used in this algebraic step was that they do not have any finite critical orbit. Since $p_{c_i}$ ($i=1,2$) is unicritical, this condition is equivalent to requiring that $p_{c_i}$ is not post-critically finite. This observation leads to the following interesting corollary.

\begin{corollary}[Injectivity of Unicritical Renormalization Operator]\label{injectivity_renorm}
Let $c_1,\ c_2 \in \mathcal{M}_d$ be such that suitable iterates of $p_{c_1}$ and $p_{c_2}$ admit unicritical polynomial-like restrictions (renormalizations) $$\mathcal{R}p_{c_1}:= p_{c_1}^{\circ k_1} : U_1\to V_1,\ \textrm{and}\ \mathcal{R}p_{c_2}:= p_{c_2}^{\circ k_2}: U_2\to V_2.$$ Assume further that $p_{c_i}$ ($i=1,2$) are not post-critically finite. If $\mathcal{R}p_{c_1}$ and $\mathcal{R}p_{c_2}$ are conformally conjugate, then $p_{c_1}$ and $p_{c_2}$ are affinely conjugate.
\end{corollary}

Using essentially the same ideas, one can prove a variant of the above result for polynomials of arbitrary degree, provided that the parabolic point has exactly one petal, and its immediate basin of attraction contains exactly one critical point (of possibly higher multiplicity).

\begin{proposition}[Unicritical Basins]
Let $p_1$ and $p_2$ be two polynomials (of any degree) satisfying $p_i(0)=0$, and $p_i(z) = z+z^2+\mathcal{O}(\vert z\vert^3)$ locally near $0$. Let $U_i$ be the immediate basin of attraction of $p_i$ at $0$, and assume that $p_i$ has exactly one critical point of multiplicity $k_i$ in $U_i$. If $p_1$ and $p_2$ are (locally) conformally conjugate in some neighborhoods of $0$, then $k_1=k_2$, and there exist polynomials $h$, $h_1$ and $h_2$ such that $p_1 \circ h_1 = h_1 \circ h$, $p_2 \circ h_2 = h_2 \circ h$. In particular, $\Deg (p_1)=\Deg (p_2)$.
\end{proposition}

Let us now proceed to the proof of Theorem~\ref{Parabolic_Germs_Determine_Roots_Co_Roots}. Thanks to Lemma~\ref{Parabolic_Germs_Determine_Roots}, it only remains to show that if $c_1$ and $c_2$ are root points of \emph{satellite} hyperbolic components of period $n_i$ of $\mathcal{M}_{d_i}$ (i.e.,\ if $c_i\in\mathcal{M}_{d_i}^{\textrm{par}}\setminus\mathcal{M}_{d_i}^{\textrm{cusp}}$, $i=1,2$) such that the parabolic germs determined by $p_{c_1}^{\circ n_1}$ and $p_{c_2}^{\circ n_2}$ near their characteristic parabolic points are conformally conjugate, then $p_{c_1}$ and $p_{c_2}$ are affinely conjugate. It is instructive to mention that the principal technical difference between the primitive and satellite cases is that unlike in the primitive situation, the map $p_{c_i}$ has multiple attracting petals at its characteristic parabolic point, and hence $p_{c_i}^{\circ n_i}$ does \emph{not} restrict to a polynomial-like map in a neighborhood of the closure of its characteristic Fatou component. Therefore, in order to implement our general strategy of promoting a parabolic germ conjugacy to a conformal conjugacy between suitable polynomial-like restrictions (of $p_{c_1}^{\circ n_1}$ and $p_{c_2}^{\circ n_2}$), one needs to work with a different (and somewhat more complicated) polynomial-like restriction of $p_{c_i}^{\circ n_i}$.

\begin{proof}[Proof of Theorem~\ref{Parabolic_Germs_Determine_Roots_Co_Roots}]
The number of attracting petals of a parabolic germ is a topological conjugacy invariant. If the parabolic cycles of the polynomials $z^{d_1}+c_1$ and $z^{d_2}+c_2$ have a single attracting petal, then the period of the characteristic parabolic point of $z^{d_i}+c_i$ ($i=1,2$) coincides with the period of the characteristic Fatou component. Hence, we are in the case of Lemma~\ref{Parabolic_Germs_Determine_Roots}, and therefore, $d_1=d_2$, and $p_{c_1} $ and $p_{c_2}$ are affinely conjugate.

Henceforth, we assume that $c_1$ and $c_2$ are roots of some satellite components of $\mathcal{M}_{d_1}$ and $\mathcal{M}_{d_2}$ respectively. Let the period of the parabolic cycle of $p_{c_i}$ be $k_i$ (so $c_i$ sits on the boundary of a hyperbolic component of period $k_i$ and a hyperbolic component of period $n_i$). We denote the characteristic Fatou component of $p_{c_i}$ by $U_{c_i}$. Set $q_i := n_i/k_i$. It is easy to verify that  the Taylor series expansion of $p_{c_i}^{\circ n_i}$ at $z_i$ is given by
$$
p_{c_i}^{\circ n_i}(z)=z+b_i(z-z_i)^{q_i+1}+O\left((z-z_i)^{q_i+2}\right)
$$ 
for some $b_i \in \mathbb{C}^*$. In fact, the number of attracting petals of $p_{c_i}^{\circ n_i}$ at $z_i$ is $q_i$.

If the parabolic germs of $p_{c_i}^{\circ n_i}$ (for $i=1, 2$) are conformally conjugate, then they must have the same number of attracting petals at the characteristic parabolic point $z_i$; i.e.,\ $q_1=q_2$. We set $q:=q_1=q_2>1$. Moreover, these petals are permuted transitively by $p_{c_i}^{\circ k_i}$. Furthermore, by looking at the ramification index of the unique singular value of their common horn maps, we deduce that $d_1=d_2$. We will denote this common degree by $d$.

\begin{figure}[ht!]
\includegraphics[width=0.6\linewidth]{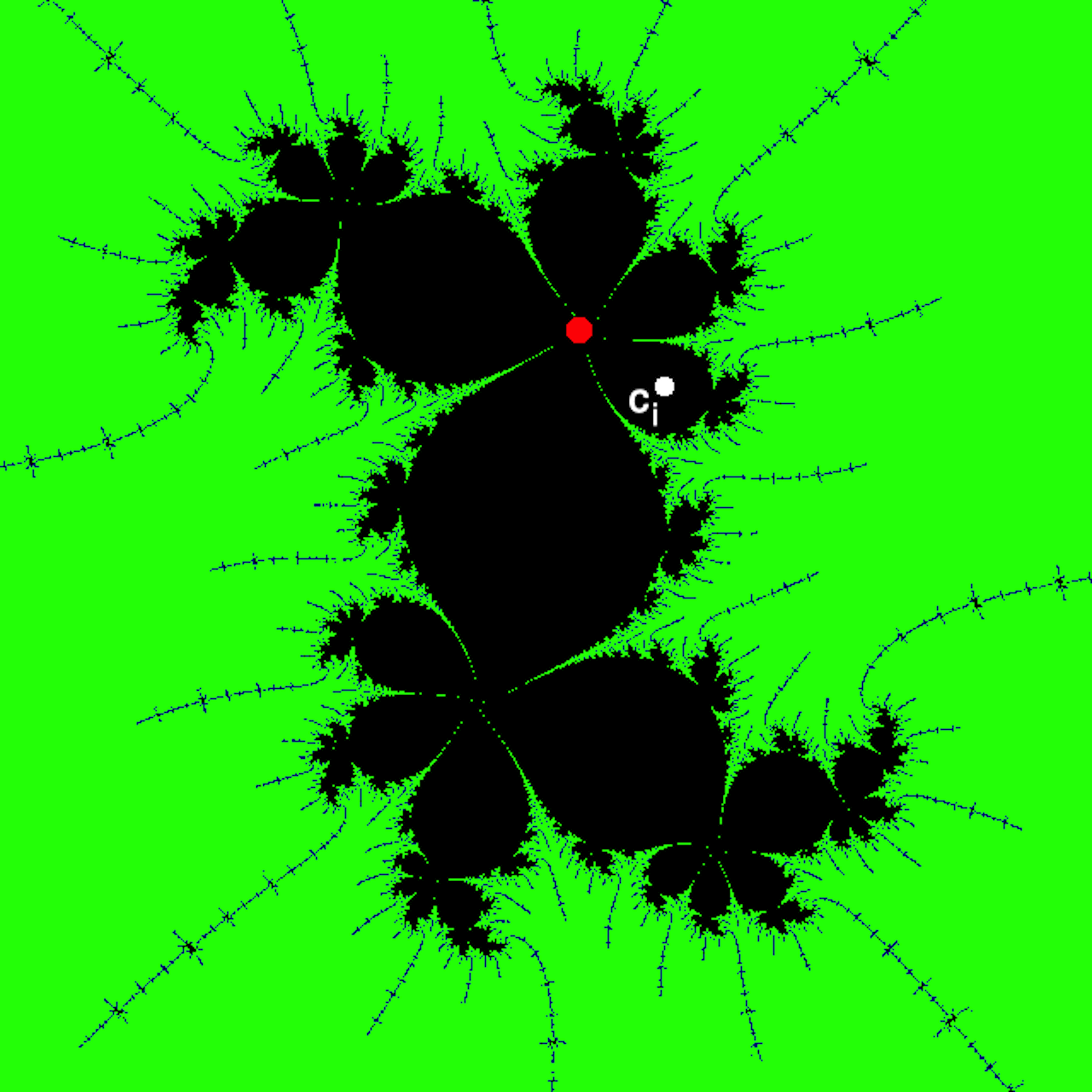}
\caption{The figure shows a part of the dynamical plane of the root $c_i$ of a satellite hyperbolic component (of the Mandelbrot set) with $5$ attracting petals touching at the characteristic parabolic point (marked in red) of $p_{c_i}$. The filled Julia set of the polynomial-like restriction $\tau_i$ of $p_{c_i}^{\circ n_i}$ is also shown.}
\label{poly_like_rabbit_fig}
\end{figure}

As in the proof of Lemma~\ref{Parabolic_Germs_Determine_Roots}, we can post-compose the conformal conjugacy between the germs $p_{c_1}^{\circ n_1}\vert_{N_{z_1}}$ and $p_{c_2}^{\circ n_2}\vert_{N_{z_2}}$ (where $N_{z_i}$ is a sufficiently small neighborhood of $z_i$) with a suitable iterate of $p_{c_2}^{\circ k_2}\vert_{N_{z_2}}$ to require that the germ conjugacy sends $p_{c_1}^{\circ rn_1}(c_1)$ to $p_{c_2}^{\circ rn_2}(c_2)$ for $r$ large enough (compare \cite[Lemma~4.2]{LoMu}). One can now carry out the arguments of Lemma~\ref{umbilical_cubic} to conclude that $\eta$ extends to a conformal conjugacy between restrictions of $p_{c_1}^{\circ n_1}$ and $p_{c_2}^{\circ n_2}$ on some neighborhoods of $\overline{U_{c_1}}$ and $\overline{U_{c_2}}$ (respectively).

Note that the restriction of the polynomial $p_{c_i}^{\circ n_i}$ to a neighborhood of $\overline{U_{c_i}}$ is not polynomial-like. However, $p_{c_i}^{\circ k_i}$, and hence $p_{c_i}^{\circ n_i}$, admits a polynomial-like restriction $\tau_i$ (such that the domain of the polynomial-like restriction contains all the periodic Fatou components touching at the characteristic parabolic point of $p_{c_i}$) that is hybrid equivalent to some map $p_{c_i'}$ with a fixed point of multiplier $e^{\frac{2\pi ir_i}{q}}$ (see Figure~\ref{poly_like_rabbit_fig}). The conclusion of the previous paragraph, combined with the arguments of \cite[Lemma~4.1]{LoMu}, implies that these polynomial-like maps $\tau_1$ and $\tau_2$ are conformally conjugate. It now follows from Corollary~\ref{injectivity_renorm} that $p_{c_1}$ and $p_{c_2}$ are affinely conjugate. 
\end{proof}

The fundamental factor that makes the above proofs work is unicriticality since one can read off the conformal position of the unique critical value from the extended horn map. The next best family of polynomials, where this philosophy can be applied, is $\{ f_c^{\circ 2}\}_{c\in \mathbb{C}}$. The proof of rigidity of parabolic parameters of $\mathcal{M}_d^*$ comes in two different flavors. The fact that the even period parabolic parameters of $\mathcal{M}_d^*$ are completely determined (up to affine conjugacy) by their parabolic germs follows by an argument similar to the one employed in the proof of Theorem~\ref{Parabolic_Germs_Determine_Roots_Co_Roots}. However, the case of odd period non-cusp parabolic parameters is slightly more tricky since for such a parabolic parameter $c$ (of parabolic orbit period $k$), the characteristic parabolic germs of $f_c^{\circ 2k}$ and $f_{c^\ast}^{\circ 2k}$ are always conformally conjugate by the local biholomorphism $\iota\circ f_c^{\circ k}$.

\begin{proof}[Proof of Theorem~\ref{recovering_Anti-polynomials}]
Let $U_{c_i}$ be the characteristic Fatou component of $f_{c_i}$. Note that by \cite[Proposition~4]{BE}, if $c_1\in\Omega_{d_1}^{\mathrm{even}}$, then the corresponding (upper) extended horn map(s) has (have) exactly one singular value. On the other hand, if $c_1\in \Omega_{d_1}^{\mathrm{odd}}$, then the corresponding (upper) extended horn map(s) has (have) exactly two distinct singular values. Since the parabolic germs of $f_{c_1}^{\circ 2n_1}$ and $f_{c_2}^{\circ 2n_2}$ are conformally conjugate, they have common (upper) extended horn map(s). By looking at the number of singular values of the common extended horn map(s), and their ramification indices, we conclude that

i) $d_1=d_2=d$ (say), and
 
ii) either both the $c_i$ are in $\Omega_{d}^{\mathrm{odd}}$, or both the $c_i$ are in $\Omega_{d}^{\mathrm{even}}$.

If both $c_i$ are in $\Omega_{d}^{\mathrm{even}}$, then the first holomorphic return maps of $U_{c_i}$ are conformally conjugate (in fact, they are conjugate to the same Blaschke product on $\mathbb{D}$). Therefore arguments similar to the ones employed in the proofs of Lemma~\ref{Parabolic_Germs_Determine_Roots} and Theorem~\ref{Parabolic_Germs_Determine_Roots_Co_Roots} show that suitable polynomial-like restrictions of $f_{c_1}^{\circ 2n_1}$ and $f_{c_2}^{\circ 2n_2}$ are conformally conjugate. We can now invoke \cite[Theorem~1]{I2} to deduce the existence of polynomials $h$, $h_1$ and $h_2$ such that $f_{c_1}^{\circ 2n_1} \circ h_1 = h_1 \circ h$, $f_{c_2}^{\circ 2n_2}\circ h_2 = h_2 \circ h$. In particular, we have that $d^{2n_1}=\Deg (f_{c_1}^{\circ 2n_1})=\Deg (f_{c_2}^{\circ 2n_2})=d^{2n_2}$. Hence, $n_1=n_2$. Finally, applying Ritt and Engstrom's reduction steps (similar to the proof of Lemma~\ref{Parabolic_Germs_Determine_Roots}), we can conclude from the semi-conjugacy relations that $f_{c_1}$ and $f_{c_2}$ are affinely conjugate. Hence, $c_2=c_1$ in $\Omega_d^{\mathrm{even}}/\mathord\sim$.

The case when both $c_i$ are in $\Omega_{d}^{\mathrm{odd}}$ is more delicate because the conformal conjugacy class of $f_{c_i}^{2n_1}\vert_{U_{c_i}}$ depends on the critical {\'E}calle height of $f_{c_i}$. We assume that $g_1:=f_{c_1}^{\circ 2n_1}\vert_{N_{z_1}}$ and $g_2:=f_{c_2}^{\circ 2n_2}\vert_{N_{z_2}}$ are conformally conjugate by some local biholomorphism $\phi:N_{z_1}\to N_{z_2}$ (where $N_{z_i}$ is a sufficiently small neighborhood of $z_i$). Then these two germs have the same horn map germ at $0$, and hence $f_{c_1}^{\circ 2n_1}$ and $f_{c_2}^{\circ 2n_2}$ have the same extended horn map $h^+$ at $0$ (recall that the domain of $h^+$ is its maximal domain of analyticity; i.e.,\ $h^+$ is completely determined by the germ of the horn map at $0$). Let $\psi^{\mathrm{att}}_{c_2}$ be an extended attracting Fatou coordinate for $f_{c_2}^{\circ 2n_2}$ at $z_2$, normalized so that the attracting equator maps to the real line, and $\psi^{\mathrm{att}}_{c_2}(c_2)=it$, for some $t\in\R$. Then, there exists an extended attracting Fatou coordinate $\psi^{\mathrm{att}}_{c_1}$ for $f_{c_1}^{\circ 2n_1}$ at $z_1$ such that $\psi^{\mathrm{att}}_{c_1}=\psi^{\mathrm{att}}_{c_2}\circ \phi$ in their common domain of definition. By \cite[Proposition~4]{BE}, $h^+$ is a ramified covering with exactly two critical values. This implies that 
$$
\left\{ \Pi(\psi^{\mathrm{att}}_{c_1}(c_1)), \Pi(\psi^{\mathrm{att}}_{c_1}(f_{c_1}^{\circ n_1}(c_1)))\right\} =\left\{\Pi(\psi^{\mathrm{att}}_{c_2}(c_2)), \Pi(\psi^{\mathrm{att}}_{c_2}(f_{c_2}^{\circ n_2}(c_2)))\right\}.
$$  

\noindent We now consider two cases.
\bigskip

\noindent \textbf{Case 1: $\Pi(\psi^{\mathrm{att}}_{c_1}(c_1))= \Pi(\psi^{\mathrm{att}}_{c_2}(c_2))$.} \quad
We can assume, possibly after modifying the conformal conjugacy $\phi$ (as in the proof of Lemma~\ref{Parabolic_Germs_Determine_Roots}) that
\begin{align*}
\psi^{\mathrm{att}}_{c_1} &= \psi^{\mathrm{att}}_{c_2}\circ \phi,\ \psi^{\mathrm{att}}_{c_1}(c_1)=it=\psi^{\mathrm{att}}_{c_2}(c_2).
\end{align*}

\noindent Since $\psi^{\mathrm{att}}_{c_2}$ maps the attracting equator (at $z_2$) to the real line, it conjugates $f_{c_2}^{\circ n_2}$ to the map $\zeta\mapsto\overline{\zeta}+1/2$. Hence, $\psi^{\mathrm{att}}_{c_2}(f_{c_2}^{\circ n_2}(c_2))=1/2-it$. On the other hand, $\psi^{\mathrm{att}}_{c_1}$ conjugates $f_{c_1}^{\circ 2n_1}$ to the translation $\zeta\mapsto\zeta+1$, and hence must conjugate $f_{c_1}^{\circ n_1}$ to a map of the form $\zeta\mapsto\overline{\zeta}+1/2+i\beta$, for some $\beta\in\R$ (compare \cite[Lemma~2.3]{HS}). Thus, we have that $\psi^{\mathrm{att}}_{c_1}(f_{c_1}^{\circ n_1}(c_1))=1/2-it+i\beta$. However, by our assumption, $\Pi(\psi^{\mathrm{att}}_{c_1}(f_{c_1}^{\circ n_1}(c_1)))=\Pi(\psi^{\mathrm{att}}_{c_2}(f_{c_2}^{\circ n_2}(c_2)))$, and hence, $\beta=0$. This shows that
$$
\psi^{\mathrm{att}}_{c_1}(f_{c_1}^{\circ n_1}(c_1))=1/2-it=\psi^{\mathrm{att}}_{c_2}(f_{c_2}^{\circ n_2}(c_2)).
$$

In particular, $f_{c_1}$ and $f_{c_2}$ have equal critical {\'E}calle height $t$, and hence $f_{c_1}^{\circ 2n_1}\vert_{U_{c_1}}$ and $f_{c_2}^{\circ 2n_2}\vert_{U_{c_2}}$ are conformally conjugate. Moreover, since the germ conjugacy $\phi$ respects both the infinite critical orbits of $f_{c_i}^{\circ 2n_i}\vert_{U_{c_i}}$, we can argue as in Lemma~\ref{umbilical_cubic} to see that $\phi$ extends to a conformal conjugacy between $f_{c_1}^{\circ 2n_1}$ and $f_{c_2}^{\circ 2n_2}$ restricted to some neighborhoods of $\overline{U_{c_i}}$. Since $c_i$ is a non-cusp parameter, $z_i$ has exactly one attracting petal, and hence $\phi$ induces a conformal conjugacy between the polynomial-like restrictions of $f_{c_1}^{\circ 2n_1}$ and $f_{c_2}^{\circ 2n_2}$ in some neighborhoods of $\overline{U_{c_1}}$ and $\overline{U_{c_2}}$ respectively. As in the even period case, we can now appeal to \cite[Theorem~1]{I2} and apply Ritt and Engstrom's reduction steps (similar to the proof of Lemma~\ref{Parabolic_Germs_Determine_Roots}) to conclude that $f_{c_1}$ and $f_{c_2}$ are affinely conjugate.

\noindent \textbf{Case 2: $\Pi(\psi^{\mathrm{att}}_{c_1}(c_1))= \Pi(\psi^{\mathrm{att}}_{c_2}(f_{c_2}^{\circ n_2}(c_2)))$.} \quad
Since $\phi$ is a conformal conjugacy between $f_{c_1}^{\circ 2n_1}\vert_{N_{z_1}}$ and $f_{c_2}^{\circ 2n_2}\vert_{N_{z_2}}$, the map $\widetilde{\phi}:=\iota\circ f_{c_2}^{\circ n_2}\circ\phi$ is a conformal conjugacy between the characteristic parabolic germs of $f_{c_1}^{\circ 2n_1}$ and $f_{c_2^*}^{\circ 2n_2}$. Let $\psi^{\mathrm{att}}_{c_2^*}$ be an extended attracting Fatou coordinate for $f_{c_2^*}$ at $z_2^*$ such that $\psi^{\mathrm{att}}_{c_2^*}\circ \iota\circ f_{c_2}^{\circ n_2}=\psi^{\mathrm{att}}_{c_2}$ in their common domain of definition. Therefore, 
\begin{align*}
\psi^{\mathrm{att}}_{c_1}
&=\psi^{\mathrm{att}}_{c_2}\circ \phi
\\
&=\psi^{\mathrm{att}}_{c_2^*}\circ \iota\circ f_{c_2}^{\circ n_2}\circ \phi
\\
&= \psi^{\mathrm{att}}_{c_2^*}\circ \widetilde{\phi}
\end{align*}
in their common domain of definition. 

Moreover, a simple computation shows that 
\begin{align*}
\Pi(\psi^{\mathrm{att}}_{c_1}(c_1))
&= \Pi(\psi^{\mathrm{att}}_{c_2^*}(c_2^*)).
\end{align*}
The situation now reduces to that of Case 1, and a similar argumentation shows that $f_{c_1}$ and $f_{c_2^*}$ are affinely conjugate.

Combining Case 1 and Case 2, we conclude that $c_2\in \lbrace c_1, c_1^*\rbrace$ in $\Omega_d^{\mathrm{odd}}/\mathord\sim$.
\end{proof}

\section{Polynomials with Real-Symmetric Parabolic Germs}\label{real-symmetric_parabolic_germs}

In this section, we will discuss another local-global principle for parabolic germs that are obtained by restricting a polynomial map of the plane near a parabolic fixed/periodic point. Recall that a parabolic germ $g$ at $0$ is said to be \emph{real-symmetric} if in some conformal coordinates, $\overline{g(\overline{z})}=g(z)$; i.e.,\ if all the coefficients in its power series expansion are real after a local conformal change of coordinates. This is a strong local condition, and we believe that in general, a polynomial parabolic germ can be real-symmetric only if the polynomial itself has a global anti-holomorphic involutive symmetry. 

By \cite[Corollary~4.8]{IM4}, if $f_c(z)=\overline{z}^d+c$ has a simple (exactly one attracting petal) parabolic orbit of odd period, and if the critical {\'E}calle height is $0$, then the corresponding parabolic germ is real-symmetric if and only if $f_c$ commutes with a global anti-holomorphic involution. In this section, we generalize this result, and also prove the corresponding theorem for unicritical holomorphic polynomials.

We will make use of our discussion on extended horn maps in Section~\ref{cauliflower_recover}. The following characterization of real-symmetric parabolic germs, and the symmetry of its upper and lower horn maps will be useful for us. The result is classical \cite[\S 2.8.4]{FL}.

\begin{lemma}\label{real_germs}
For a simple parabolic germ $g$, the following are equivalent:
\begin{itemize}
\item $g$ is a real-symmetric germ,

\item there is a $g$-invariant real-analytic curve $\Gamma$ passing through the parabolic fixed point of $g$,

\item there is an anti-holomorphic involution $\widetilde{\iota}$ defined in a neighborhood of the parabolic fixed point (and fixing it) of $g$ such that $g$ commutes with $\widetilde{\iota}$.

If any of these equivalent conditions are satisfied, one can choose attracting and repelling Fatou coordinates for $g$ such that the involution $w\mapsto 1/\overline{w}$ is a conjugacy between the upper and lower horn map germs $h^+$ and $h^-$; i.e.,\ $1/\overline{h^-\left(1/\overline{w}\right)}=h^+(w)$ for $w$ near $0$.
\end{itemize} 
\end{lemma}

In fact, the statements about the horn map germs $h^{\pm}$ (near $0$ and $\infty$ respectively) can be made somewhat more global. 

\begin{lemma}[Extended Horn Maps for Real-symmetric Germs]\label{extended_horns_symmetric}
Let $p$ be a polynomial with a simple parabolic fixed point $z_0$ such that the parabolic germ of $p$ at $z_0$ is real-symmetric. If we normalize the attracting and repelling Fatou coordinates of $p$ at $z_0$ such that they map the real-analytic curve $\Gamma$ to the real line, then the following is true for the corresponding horn maps: $\mathcal{D}_0^-$ is the image of $\mathcal{D}_0^+$ under $w\mapsto 1/\overline{w}$, and $1/\overline{h^-\left(1/\overline{w}\right)}=h^+(w)$ for all $w\in \mathcal{D}_0^+$.
\end{lemma}

\begin{proof}
This follows from Lemma~\ref{real_germs}, and the identity principle for holomorphic maps (since the extended horn maps are the maximal analytic continuations of the horn map germs).
\end{proof}

\begin{definition}
We say that $p_c$ (respectively $f_c$) is a \emph{real} polynomial (respectively anti-polynomial) if $p_c$ (respectively $f_c$) commutes with an anti-holomorphic involution of the plane.
\end{definition}
 
We now prove Theorem~\ref{Real_Germs_Real_Parameters}, which is another local-global principle for unicritical holomorphic polynomials with parabolic cycles. Recall that $\mathcal{M}_d^{\mathrm{par}}$ is the set of all parabolic parameters of $\mathcal{M}_d$. For $c\in \mathcal{M}_d^{\mathrm{par}}$, let $z_c$ be the characteristic parabolic point, and $U_c$ be the characteristic Fatou component (of period $n$) of $p_{c}(z)=z^d+c$. 

\begin{proof}[Proof of Theorem~\ref{Real_Germs_Real_Parameters}]
We assume that the parabolic germ of $g:=p_c^{\circ n}$ at $z_c$ is real-symmetric. Let $\alpha$ be a local conformal conjugacy between $g$, and a real germ $h$ fixing $0$. Observe that $\iota : z \mapsto z^*$ is an anti-holomorphic conjugacy between $p_c$ and $p_{c^*}$. It is easy to check that the germ $\iota \circ g\circ \iota= p_{c^*}^{\circ n}$ at $z_c^*$ is also real-symmetric, and the local biholomorphism $\iota \circ \alpha \circ \iota$ conjugates the parabolic germ $\iota \circ g\circ \iota$ at $z_c^*$ to the same real parabolic germ $h$ as obtained above. Thus, the parabolic germs $g$ at $z_c$, and $\iota\circ g\circ\iota$ at $z_c^*$ are conformally conjugate by $\eta :=\left(\iota \circ \alpha \circ \iota\right)^{-1}\circ\alpha$. Therefore, by Theorem~\ref{Parabolic_Germs_Determine_Roots_Co_Roots}, the maps $p_c$ and $p_{c^*}$ are affinely conjugate. A straightforward computation now shows that $c^*=\omega^j c$ where $\omega=\exp(\frac{2\pi i}{d-1})$, and $j \in \mathbb{N}$. But this precisely means that $p_c$ commutes with the global anti-holomorphic involution $\zeta \mapsto \omega^{-j} \zeta^*$.
\end{proof}

Finally, let us record the analogue of Theorem~\ref{Real_Germs_Real_Parameters} in the unicritical anti-holomorphic family. The following theorem also sharpens \cite[Corollary~4.8]{IM4}. We continue with the terminologies introduced in the previous section.

\begin{proof}[Proof of Theorem~\ref{Real_Germs_Real_Parameters_Anti}]
The case when $c\in \Omega_{d_i}^{\mathrm{even}}$ is similar to the holomorphic case (Theorem~\ref{Real_Germs_Real_Parameters}). By a completely similar argument using Theorem~\ref{recovering_Anti-polynomials}, we can conclude that $c^*=\omega^j c$ for some $j\in \lbrace 0, 1, \cdots, d\rbrace$, where $\omega=\exp(\frac{2\pi i}{d+1})$. But this is equivalent to saying that $f_c$ commutes with the global anti-holomorphic involution $\zeta \mapsto \omega^{-j} \zeta^*$.

Now we focus on the case $c\in \Omega_{d_i}^{\mathrm{odd}}$. Note that in this case, the invariant real-analytic curve $\Gamma$ passing through $z_c$ (compare  Lemma~\ref{real_germs}) is simply the union of the attracting equator at $z_c$, the parabolic point $z_c$, and the repelling equator at $z_c$. By \cite[Lemma~2.3]{HS}, we can choose an attracting Fatou coordinate $\psi^{\mathrm{att}}_{c}$ for the first return map $f_{c}^{\circ n}$ on the attracting petal $\mathcal{P}^{\textrm{att}}_c\subset U_c$ such that $$\re(\psi^{\mathrm{att}}_{c}(c))=0,\ \textrm{and}\ \psi^{\mathrm{att}}_{c}(f_c^{\circ n}(z))=\overline{\psi^{\mathrm{att}}_{c}(z)}+1/2,$$ for $z\in\mathcal{P}^{\textrm{att}}_c$. We then have $$\psi^{\mathrm{att}}_{c}(c)=it,$$ where $t\in\R$ is the critical {\'E}calle height of $f_c$. Since $\psi^{\mathrm{att}}_{c}$ conjugates $f_c^{\circ n}$ to $\zeta\mapsto\overline{\zeta}+1/2$, it follows that $$\psi^{\mathrm{att}}_{c}(f_{c}^{\circ n}(c))=1/2-it.$$

By construction of $\psi^{\mathrm{att}}_{c}$, it maps the attracting equator to the real line. Moreover, we can choose a repelling Fatou coordinate at $z_c$ such that it maps the repelling equator to the real line (once again by \cite[Lemma~2.3]{HS}). With such choice of Fatou coordinates at $z_c$, the extended upper and lower horn maps of $f_{c}^{\circ 2n}$ at $z_c$ are conjugated by $w\mapsto 1/\overline{w}$ (by Lemma~\ref{extended_horns_symmetric}). In particular, we have
\begin{align*}
\{ \Pi(\psi^{\mathrm{att}}_{c}(c)), \Pi(\psi^{\mathrm{att}}_{c}(f_{c}^{\circ n}(c)))\} &=\{1/\overline{\Pi(\psi^{\mathrm{att}}_{c}(c)}), 1/\overline{\Pi(\psi^{\mathrm{att}}_{c}(f_{c}^{\circ n}(c)))}\}.
\end{align*}  

Now a simple computation using the relations $$\psi^{\mathrm{att}}_{c}(c)=it,\ \textrm{and}\ \psi^{\mathrm{att}}_{c}(f_{c}^{\circ n}(c))=\frac{1}{2}-it$$ shows that we must have $\Pi(\psi^{\mathrm{att}}_{c}(c))=1/\overline{\Pi(\psi^{\mathrm{att}}_{c}(c)})$, and hence $t=0$. Therefore, $c$ is a critical {\'E}calle height $0$ parameter. 

Now as in the even period case, there exists a local conformal conjugacy $\alpha$ conjugating the germ of $f_c^{\circ 2n}$ at $z_c$ to a real germ $h$ fixing $0$. Therefore, the local biholomorphism $\iota \circ \alpha \circ \iota$ conjugates the parabolic germ of $f_{c^*}^{\circ 2n}=\iota \circ f_c^{\circ 2n}\circ \iota$ at $z_{c^*}$ to the same real parabolic germ $h$ obtained above. It follows that the germ of $f_{c}^{\circ 2n}$ at $z_{c}$ and the germ of $f_{c^*}^{\circ 2n}$ at $z_{c^*}$ are conformally conjugate via $\eta :=\left(\iota \circ \alpha \circ \iota\right)^{-1}\circ\alpha$, and $\eta$ preserves the corresponding dynamically marked critical orbits (here we have used the fact that $c$ is a critical {\'E}calle height $0$ parameter). Choosing an extended attracting Fatou coordinate $\psi^{\mathrm{att}}_{c^*}$ for $f_{c^*}^{\circ n}$ at $z_{c^*}$ (normalized so that the attracting equator maps to the real line), we can find an extended attracting Fatou coordinate $\psi^{\mathrm{att}}_{c}$ for $f_{c}^{\circ n}$ at $z_c$ such that $\psi^{\mathrm{att}}_{c}=\psi^{\mathrm{att}}_{c^*}\circ \eta$ in their common domain of definition. Moreover, by our construction of $\eta$, we have that $\Pi(\psi^{\mathrm{att}}_{c}(c))= \Pi(\psi^{\mathrm{att}}_{c^*}(c^*)).$ It now follows from (Case 1 of the proof of) Theorem~\ref{recovering_Anti-polynomials} that $c^*=\omega^j c$ for some $j\in \lbrace 0, 1, \cdots, d\rbrace$, where $\omega=\exp(\frac{2\pi i}{d+1})$. Therefore, $f_c$ commutes with the global anti-holomorphic involution $\zeta \mapsto \omega^{-j} \zeta^*$.
\end{proof}
\begin{remark}
It follows from the proof of the above theorem that if an odd period non-cusp parabolic parameter of $\mathcal{M}_d^*$ has a real-symmetric parabolic germ, then it must be a critical {\'E}calle height $0$ parameter. This is another example where a global feature of the dynamics can be read off from its local properties.
\end{remark}


\begin{thebibliography}{CHRSC89}

\bibitem[AKLS09]{L4}
A.~Avila, J.~Kahn, M.~Lyubich, and W.~Shen.
\newblock Combinatorial rigidity for unicritical polynomials.
\newblock {\em Ann. of Math. (2)}, 170:783--797, 2009.

\bibitem[BE02]{BE}
X.~Buff and A.~L. Epstein.
\newblock A parabolic {P}ommerenke-{L}evin-{Y}occoz inequality.
\newblock {\em Fund. Math.}, 172:249--289, 2002.

\bibitem[Ch{\'e}21]{Ch}
A.~Ch{\'e}ritat.
\newblock Near parabolic renormalization for unicritical holomorphic maps.
\newblock \url{https://arxiv.org/abs/1404.4735}, to appear in {\em Arnold Math.
  J.}, 2021.

\bibitem[CHRSC89]{CHRS}
W.~D. Crowe, R.~Hasson, P.~J. Rippon, and P.~E.~D. Strain-Clark.
\newblock On the structure of the {M}andelbar set.
\newblock {\em Nonlinearity}, 2, 1989.

\bibitem[DH85a]{DH2}
A.~Douady and J.~H. Hubbard.
\newblock On the dynamics of polynomial-like mappings.
\newblock {\em Ann. Sci. {\'E}c. Norm. Sup{\'e}r. (4)}, 18:287--343, 1985.

\bibitem[DH85b]{orsay}
Adrien Douady and John~H. Hubbard.
\newblock {\em {\'E}tude dynamique des polyn\^omes complexes {I}, {II}}.
\newblock Publications Math\'ematiques d'Orsay. Universit\'e de Paris-Sud,
  D\'epartement de Math\'ematiques, Orsay, 1984 - 1985.

\bibitem[EMS16]{EMS}
D.~Eberlein, S.~Mukherjee, and D.~Schleicher.
\newblock Rational parameter rays of the multibrot sets.
\newblock In {\em Dynamical Systems, Number Theory and Applications},
  chapter~3, pages 49--84. World Scientific, 2016.

\bibitem[Eng41]{Eng}
H.~T. Engstrom.
\newblock Polynomial substitutions.
\newblock {\em Amer. J. Math.}, 63:249--255, 1941.

\bibitem[G{\'S}97]{GS}
J.~Graczyk and G.~{\'S}wiatek.
\newblock Generic hyperbolicity in the logistic family.
\newblock {\em Ann. of Math. (2)}, 146:1--52, 1997.

\bibitem[HS14]{HS}
J.~H. Hubbard and D.~Schleicher.
\newblock {M}ulticorns are not path connected.
\newblock In {\em Frontiers in Complex Dynamics: In Celebration of John
  Milnor's 80th Birthday}, pages 73--102. Princeton University Press, 2014.

\bibitem[IK12]{IK}
H.~Inou and J.~Kiwi.
\newblock {C}ombinatorics and topology of straightening maps, {I}:
  {C}ompactness and bijectivity.
\newblock {\em Adv. Math.}, 231:2666--2733, 2012.

\bibitem[IM16]{IM}
H.~Inou and S.~Mukherjee.
\newblock Non-landing parameter rays of the multicorns.
\newblock {\em Invent. Math.}, 204:869--893, 2016.

\bibitem[IM21]{IM4}
H.~Inou and S.~Mukherjee.
\newblock Discontinuity of straightening in anti-holomorphic dynamics: I.
\newblock {\em Trans. Amer. Math. Soc.}, 374:6445--6481, 2021.

\bibitem[Ino09]{I}
H.~Inou.
\newblock {C}ombinatorics and topology of straightening maps {II}:
  {D}iscontinuity.
\newblock \url{http://arxiv.org/abs/0903.4289}, 2009.

\bibitem[Ino11]{I2}
H.~Inou.
\newblock Extending local analytic conjugacies.
\newblock {\em Trans. Amer. Math. Soc.}, 363:331--343, 2011.

\bibitem[Ino19]{I1}
H.~Inou.
\newblock Self-similarity for the {T}ricorn.
\newblock {\em Exp. Math.}, 28:440--455, 2019.

\bibitem[LLMM21]{LLMM2}
S.-Y. Lee, M.~Lyubich, N.~G. Makarov, and S.~Mukherjee.
\newblock {S}chwarz reflections and the {T}ricorn.
\newblock \url{https://arxiv.org/abs/1812.01573v2}, 2021.

\bibitem[LM18]{LoMu}
L.~Lomonaco and S.~Mukherjee.
\newblock A rigidity result for some parabolic germs.
\newblock {\em Indiana Univ. Math. J.}, 67:2089--2101, 2018.

\bibitem[Lor06]{FL}
F.~Loray.
\newblock Pseudo-groupe d'une singularit{\'e} de feuilletage holomorphe en
  dimension deux.
\newblock \url{https://hal.archives-ouvertes.fr/hal-00016434/document}, 2006.

\bibitem[LY14]{LY}
O.~{Lanford III} and M.~Yampolsky.
\newblock {\em Fixed point of the parabolic renormalization operator}.
\newblock Springer International Publishing, 1st edition, 2014.

\bibitem[Mil92]{M3}
J.~Milnor.
\newblock Remarks on iterated cubic maps.
\newblock {\em Exp. Math.}, 1:5--24, 1992.

\bibitem[Mil00a]{M5}
J.~Milnor.
\newblock Local connectivity of {J}ulia sets: expository lectures.
\newblock In Tan Lei, editor, {\em The Mandelbrot Set, Theme and Variations},
  number 274 in London Mathematical Society Lecture Note Series. Cambridge
  University Press, 2000.

\bibitem[Mil00b]{M4}
J.~Milnor.
\newblock On rational maps with two critical points.
\newblock {\em Exp. Math.}, 9:333--411, 2000.

\bibitem[MNS15]{MNS}
S.~Mukherjee, S.~Nakane, and D.~Schleicher.
\newblock On {M}ulticorns and {U}nicorns {II}: bifurcations in spaces of
  antiholomorphic polynomials.
\newblock {\em Ergodic Theory Dynam. Systems}, 37:859--899, 2015.

\bibitem[Muk15]{Sa}
S.~Mukherjee.
\newblock Orbit portraits of unicritical antiholomorphic polynomials.
\newblock {\em Conform. Geom. Dyn.}, 19:35--50, 2015.

\bibitem[Nak93]{Na1}
S.~Nakane.
\newblock Connectedness of the {T}ricorn.
\newblock {\em Ergodic Theory Dynam. Systems}, 13:349--356, 1993.

\bibitem[NS03]{NS}
S.~Nakane and D.~Schleicher.
\newblock On {M}ulticorns and {U}nicorns {I} : Antiholomorphic dynamics,
  hyperbolic components and real cubic polynomials.
\newblock {\em Internat. J. Bifur. Chaos Appl. Sci. Engrg.}, 13:2825--2844,
  2003.

\bibitem[Rit22]{R}
J.~F. Ritt.
\newblock Prime and composite polynomials.
\newblock {\em Trans. Amer. Math. Soc.}, 23:51--66, 1922.

\end{thebibliography}
 
\end{document}